\documentclass[a4paper,11pt,reqno]{customart}

\usepackage[utf8x]{inputenc}
\usepackage[margin=1in]{geometry}
\usepackage{verbatim}
\usepackage{color}
\usepackage[table]{xcolor}
\usepackage{listings}
\usepackage{amssymb,amsfonts,amsthm,amsmath}
\usepackage{url}
\usepackage{enumerate}
\usepackage{todonotes}
\usepackage[all]{xy}
\usepackage{multirow}
\usepackage{attachfile}
\usepackage{stmaryrd}
\usepackage{footnote}
\makesavenoteenv{tabular}
\usepackage{hyperref}
\usepackage{cleveref}
\usepackage{upgreek}

\makeatletter
\newcommand\footnoteref[1]{\protected@xdef\@thefnmark{\ref{#1}}\@footnotemark}
\makeatother

\hypersetup{
	pdfborder={0 0 0}
}

\newcommand{\imp}{\rightarrow}

\newcommand{\Pb}{\mathbb{P}}
\newcommand{\Qb}{\mathbb{Q}}

\newcommand{\Psf}{\mathsf{P}}

\newcommand{\Acal}{\mathcal{A}}
\newcommand{\Bcal}{\mathcal{B}}
\newcommand{\Ccal}{\mathcal{C}}
\newcommand{\Dcal}{\mathcal{D}}

\newcommand{\Fcal}{\mathcal{F}}

\newcommand{\Ucal}{\mathcal{U}}
\newcommand{\Vcal}{\mathcal{V}}

\newcommand{\Ical}{\mathcal{I}}

\newcommand{\Lcal}{\mathcal{L}}
\newcommand{\Mcal}{\mathcal{M}}
\newcommand{\Ncal}{\mathcal{N}}
\newcommand{\Ocal}{\mathcal{O}}
\newcommand{\Pcal}{\mathcal{P}}

\newcommand{\Scal}{\mathcal{S}}

\newcommand{\uh}{{\upharpoonright}}

\newcommand{\wck}{\omega_1^{ck}}
\newcommand{\halt}{\font\cmsy = cmsy10 at 12pt%
                                   \hbox{\cmsy \char 59}}

\def\qt#1{``#1''}%

\newcommand{\s}[1]{\ensuremath{\sf{#1}}}

\DeclareMathOperator{\rca}{\s{RCA}_0}

\DeclareMathOperator{\wkl}{\s{WKL}}

\DeclareMathOperator{\rt}{\s{RT}}

\DeclareMathOperator{\rrt}{\s{RRT}}

\DeclareMathOperator{\ts}{\s{TS}}

\DeclareMathOperator{\fs}{\s{FS}}

\newcommand{\qvdash}{\operatorname{{?}{\vdash}}}
\newcommand{\nqvdash}{\operatorname{{?}{\nvdash}}}

\newtheoremstyle{custom}
  {10pt}
  {10pt}
  {\normalfont}
  {}
  {\bfseries}
  {}
  { }
  {}

\theoremstyle{custom}

\usepackage{xcolor}
\usepackage{soul}

\newtheorem{theorem}{Theorem}[section]
\newtheorem{lemma}[theorem]{Lemma}

\newtheorem{corollary}[theorem]{Corollary}
\newtheorem{proposition}[theorem]{Proposition}

\newtheorem*{theoremnonumber}{Theorem}

\theoremstyle{definition}
\newtheorem{definition}[theorem]{Definition}

\theoremstyle{remark}

\newtheorem {question}[theorem]{Question}

\numberwithin{equation}{section}

\title{The weakness of the pigeonhole principle under hyperarithmetical reductions}
\author{
  Benoit Monin \and Ludovic Patey
}

\begin{document}

\begin{abstract}
The infinite pigeonhole principle for 2-partitions ($\rt^1_2$) asserts the existence, for every set $A$, of an infinite subset of $A$ or of its complement. In this paper, we study the infinite pigeonhole principle from a computability-theoretic viewpoint. We prove in particular that $\rt^1_2$ admits strong cone avoidance for arithmetical and hyperarithmetical reductions. We also prove the existence, for every $\Delta^0_n$ set, of an infinite low${}_n$ subset of it or its complement. This answers a question of Wang. For this, we design a new notion of forcing which generalizes the first and second-jump control of Cholak, Jockusch and Slaman.
\end{abstract}

\maketitle


\section{Introduction}

In this paper, we study the infinite pigeonhole principle ($\rt^1_k$) from a computability-theoretic viewpoint. The infinite pigeonhole principle asserts that every finite partition of $\omega$ admits an infinite part. More formally, $\rt^1_k$ is the problem whose instances are colorings $f : \omega \to k$. An $\rt^1_k$-solution to $f$ is an infinite set $H \subseteq \omega$ such that $|f[H]| = 1$. The general question we aim to address is the following:
\begin{question}
Does every instance of $\rt^1_k$ admit a \qt{weak} solution?
\end{question}
We consider various notions of weakness, among which the inability to bound a fixed non-zero degree for the $\Delta^0_n$, arithmetical and hyperarithmetical reduction. This property is known as \emph{strong cone avoidance}. With respect to $\Delta^0_n$ and arithmetical reductions, our main theorems are:

\begin{theorem}[Main theorem 1] \label{maintheorem1}
Fix $n \geq 0$. Let $B$ be non $\emptyset^{(n)}$-computable. Every set $A$ has an infinite subset $H \subseteq A$ or $H \subseteq \overline{A}$ such that $B$ is not $H^{(n)}$-computable.
\end{theorem}

\begin{theorem}[Main theorem 2] \label{maintheorem2}
Let $B$ be non arithmetical. Every set $A$ has an infinite subset $H \subseteq A$ or $H \subseteq \overline{A}$ such that $B$ is not arithmetical in $H$.
\end{theorem}

We also study restrictions of the infinite pigeonhole principle to $\Delta^0_n$ instances. With that respect, our main theorem is:

\begin{theorem}[Main theorem 3] \label{maintheorem3}
Fix $n \geq 0$. Every $\halt^{(n+1)}$-computable set $A$ has an infinite subset $H \subseteq A$ or $H \subseteq \overline{A}$ of low${}_{n+2}$ degree.
\end{theorem}

Finally our main theorem with respect to hyperarithmetic reductions is:

\begin{theorem}[Main theorem 4] \label{maintheorem4}
Let $B$ be non hyperarithmetical. Every set $A$ has an infinite subset $H \subseteq A$ or $H \subseteq \overline{A}$ such that $B$ is not hyperarithmetical in $H$, in particular with $\omega_1^H = \wck$.
\end{theorem}

Our motivation comes from \emph{reverse mathematics}.
Reverse mathematics is a foundational program which aims to find the weakest axioms needed to prove ordinary theorems. The early reverse mathematics showed the existence of an empirical structural phenomenon, in that most theorems are provably equivalent to one among five main systems of axioms, linearly ordered by the logical implication. See Simpson's book~\cite{Simpson2009Subsystems} for a reference on reverse mathematics. However, some natural statements escape this structural phenomenon, the most famous one being \emph{Ramsey's theorem for pairs} ($\rt^2_2$). Given a set $X$, let $[X]^n$ denote the set of unordered $n$-tuples over $X$. Ramsey's theorem for $n$-tuples and $k$-colors ($\rt^n_k$) asserts the existence, for every coloring $f : [\omega]^n \to k$, of an infinite set $H \subseteq \omega$ such that $|f[\omega]^n|  = 1$. In particular, $\rt^1_k$ is the infinite pigeonhole principle.

Ramsey's theorem for pairs and two colors received a lot of attention from the computability community as it was historically the first example of statement escaping the structural phenomenon of reverse mathematics. The study of $\rt^2_k$ revealed a deep connection between the computability-theoretic features of $\rt^2_k$ and the combinatorial features of $\rt^1_k$. More precisely, almost every proof of a statement of the form \qt{Every computable instance of $\rt^2_k$ admits a weak solution} can be obtained by a proof of the statement \qt{every (arbitrary) instance of $\rt^1_k$ admits a weak solution}, with the help of very weak computability-theoretic notion called \emph{cohesiveness}. This is in particular the case for cone avoidance~\cite{Seetapun1995strength,Dzhafarov2009Ramseys}, PA avoidance~\cite{Liu2012RT22}, constant-bound trace avoidance~\cite{Liu2015Cone}, preservation of hyperimmunity~\cite{Patey2017Iterative}, and preservation of non-c.e.\ definitions~\cite{Wang2014Definability,Patey2016weakness}, among others.
In many cases, the combinatorial features of $\rt^1_k$ and the computability-theoretic features of $\rt^2_k$ can be proven to be equivalent. See Cholak and Patey~\cite[Theorem 1.5]{Cholak2019Thin} for an equivalence in the case of cone avoidance. It therefore seems essential to obtain a good understanding of the infinite pigeonhole principle in order to better understand why Ramsey's theorem for pairs escapes the structural phenomenon of reverse mathematics.

\subsection{Strong cone avoidance}

Given a partial order $\leq_r$ on $2^\omega$ and a set $X$, we let $\deg_r(X) = \{ Y : X \equiv_r Y \}$ be the \emph{degree} of $X$, where $X \equiv_r Y$ if $X \leq_r Y$ and $Y \leq_r X$. We are in particular interested in the case where $\leq_r$ is among the $\Delta^0_n$ reduction $\leq_n$, the arithmetical reduction $\leq_{arith}$ and the hyperarithmetical reduction $\leq_{hyp}$.
Given a mathematical problem $\Psf$ formulated in terms of instances and solutions, it is natural to ask which sets are \emph{$\Psf$-encodable}. Here, we say that a set $X$ is $\Psf$-encodable if there is an instance $I$ of $\Psf$ such that for every $\Psf$-solution $Y$ to $I$, $X \leq_r Y$. Some problems are very weak with respect to the order $\leq_r$, and satisfy the following property:

\begin{definition}[Strong cone avoidance]
A problem $\Psf$ \emph{admits strong cone avoidance} for $\leq_r$ if for every pair of sets $Z$ and $C$ such that $C \not \leq_r Z$, every instance $X$ of $\Psf$ admits a solution $Y$ such that $C \not \leq_r Z \oplus Y$.
\end{definition}

Dzhafarov and Jockusch~\cite{Dzhafarov2009Ramseys} proved that $\rt^1_2$ admits strong cone avoidance of the Turing reduction. Their theorem has practical applications, and yield a simpler proof of Seetapun's theorem~\cite{Seetapun1995strength}. We prove a similar result for $\Delta^0_n$ and arithmetical reductions.

\begin{theoremnonumber}[Reformulation of Main theorem 1 (Theorem \ref{maintheorem1})]
$\rt^1_2$ admits strong cone avoidance for $\Delta^0_n$ reductions.
\end{theoremnonumber}

\begin{theoremnonumber}[Reformulation of Main theorem 2 (Theorem \ref{maintheorem2})]
$\rt^1_2$ admits strong cone avoidance for arithmetical reductions.
\end{theoremnonumber}

We finally prove in the last section strong cone avoidance for hyperarithmetical reductions, the main difficulty being to show that a non-computable ordinal is never $\rt^1_2$-encodable. This gives us the following theorem:

\begin{theoremnonumber}[Reformulation of Main theorem 4 (Theorem \ref{maintheorem4})]
$\rt^1_2$ admits strong cone avoidance for hyperarithmetical reductions.
\end{theoremnonumber}

%

These theorems show the combinatorial weakness of the pigeonhole principle with respect $\rt^1_2$-encodability. To prove this, we designed a new notion of forcing with an iterated jump control generalizing the first and second jump control of Cholak, Jockusch and Slaman~\cite{Cholak2001strength}.

\subsection{Lowness and hierarchies}

The computability-theoretic study of the pigeonhole principle is also motivated by questions on the strictness of hierarchies in reverse mathematics. Some consequences of Ramsey's theorem form hierarchies of statements, parameterized by the size of the colored tuples. A first example is Ramsey's theorem itself. Indeed, $\rt^{n+1}_k$ implies $\rt^n_k$ for every $n, k \geq 1$. By the work of Jockusch~\cite{Jockusch1972Ramseys}, this hierarchy collapses starting from the triples, and by Seetapun~\cite{Seetapun1995strength}, Ramsey's theorem for pairs is strictly weaker than Ramsey's theorem for triples. We therefore have
$$
\rt^1_k < \rt^2_k < \rt^3_k = \rt^4_k = \dots
$$

Some other hierarchies have been considered in reverse mathematics.
Friedman~\cite{FriedmanFom53free} introduced the free set ($\fs^n$) and thin set theorems ($\ts^n$), while Csima and Mileti~\cite{Csima2009strength} introduced and studied the rainbow Ramsey theorem ($\rrt^n_k$). These statements are all of the form $\Psf^n$: \qt{For every coloring $f : [\omega]^n \to \omega$, there is an infinite set $H \subseteq \omega$ such that $f \uh [H]^n$ avoids some set of forbidden patterns}.
The reverse mathematics of these statements were extensively studied in the literature~\cite{Cholak2001Free,Csima2009strength,Kang2014Combinatorial,PateyCombinatorial,Patey2015Somewhere,Patey2016weakness,RiceThin,WangSome,Wang2013Rainbow,Wang2014Cohesive,Wang2014Definability,Wang2014Some}. In particular, these theorems form hierarchies which are not known to be strictly increasing.

\begin{question}\label{quest:strictness-hiearchies}
Are the hierarchies of the free set, thin set, and rainbow Ramsey theorem strictly increasing?
\end{question}

Partial results were however obtained. All these statements admit lower bounds of the form “For every $n \geq 2$, there is a computable instance of $\Psf^n$ with no $\Sigma^0_n$ solution", where $\Psf^n$ denotes any of $\rt^n_k$ (Jockusch~\cite{Jockusch1972Ramseys}), $\rrt^n_k$ (Csima and Mileti~\cite{Csima2009strength}), $\fs^n$, or $\ts^n$ (Cholak, Giusto, Hirst and Jockusch~\cite{Cholak2001Free}). From the upper bound viewpoint, all these statements follow from Ramsey's theorem. Therefore, by Cholak, Jockusch and Slaman~\cite{Cholak2001strength}, every computable instance of $\Psf^1$ admits a computable solution, and every computable instance of $\Psf^2$ admits a low${}_2$ solution. These results are sufficient to show that $\Psf^1 < \Psf^2 < \Psf^3$ in reverse mathematics. This upper bound becomes too coarse for triples. Wang~\cite{Wang2014Cohesive} proved that every computable instance of $\rrt^3_k$ admits a low${}_3$ solution. The following question is still open. A positive answer would also answer positively Question~\ref{quest:strictness-hiearchies}.

\begin{question}\label{quest:instances-hierarchies-lown}
Does every computable instance of $\fs^n$, $\ts^n$, and $\rrt^n_k$ admit a low${}_n$ solution?
\end{question}

Indeed, suppose \Cref{quest:instances-hierarchies-lown} is answered positively for some $\Psf \in \{\rrt_2, \fs, \ts\}$. For every $n$, one can iterate a relativization of \Cref{quest:instances-hierarchies-lown} to build a model $\Mcal$ of $\Psf^n$ containing only sets of low${}_n$ degree. In particular, any set in $\Mcal$ is $\Sigma^0_{n+1}$ , while by the lower bounds mentioned above, there is a computable instance of $\Psf^{n+1}$ with no $\Sigma^0_{n+1}$ solution. Thus, $\Psf^{n+1}$ fails in $\Mcal$, hence $\Psf^n$ does not imply $\Psf^{n+1}$ over $\rca$.

Upper bounds to $\fs^n$, $\ts^n$, and $\rrt^n_k$,
are usually proven inductively over $n$~\cite{Wang2014Some,PateyCombinatorial,Patey2017Iterative}, starting with the infinite pigeonhole principle for $n = 1$. In this paper, we therefore prove the following theorem, which introduces the machinery that hopefully will serve to answer positively Question~\ref{quest:instances-hierarchies-lown}.

\begin{theoremnonumber}[Main theorem 3 (\Cref{maintheorem3})]
Fix $n \geq 0$. Every $\halt^{(n+1)}$-computable set $A$ has an infinite subset $H \subseteq A$ or $H \subseteq \overline{A}$ of low${}_{n+2}$ degree.
\end{theoremnonumber}

In particular, we fully answer two questions of Wang~\cite[Questions 6.1 and 6.2]{Wang2014Cohesive}, also asked by the second author~\cite[Question 5.4]{Patey2016Open}. The cases $n = 2$ and $n = 3$ were proven by Cholak, Jockusch and Slaman~\cite{Cholak2001strength} and by the authors~\cite{Monin2018Pigeons}, respectively.

\subsection{Definitions and notation}

A \emph{binary string} is an ordered tuple of bits $a_0, \dots, a_{n-1} \in \{0, 1\}$.
The empty string is written $\epsilon$. A \emph{binary sequence} (or a \emph{real}) is an infinite listing of bits $a_0, a_1, \dots$.
Given $s \in \omega$,
$2^s$ is the set of binary strings of length $s$ and
$2^{<s}$ is the set of binary strings of length $<s$. As well,
$2^{<\omega}$ is the set of binary strings
and $2^{\omega}$ is the set of binary sequences.
Given a string $\sigma \in 2^{<\omega}$, we use $|\sigma|$ to denote its length.
Given two strings $\sigma, \tau \in 2^{<\omega}$, $\sigma$ is a \emph{prefix}
of $\tau$ (written $\sigma \preceq \tau$) if there exists a string $\rho \in 2^{<\omega}$
such that $\sigma^\frown \rho = \tau$. Given a sequence $X$, we write $\sigma \prec X$ if
$\sigma = X \uh n$ for some $n \in \omega$.
A binary string $\sigma$ can be interpreted as a finite set $F_\sigma = \{ x < |\sigma| : \sigma(x) = 1 \}$. We write $\sigma \subseteq \tau$ for $F_\sigma \subseteq F_\tau$.
We write $\#\sigma$ for the size of $F_\sigma$. Given two strings $\sigma$ and $\tau$, we let $\sigma \cup \tau$ be the unique string $\rho$ of length $\max(|\sigma|, |\tau|)$ such that $F_\rho = F_\sigma \cup F_\tau$. 

A \emph{binary tree} is a set of binary strings $T \subseteq 2^{<\omega}$ which is closed downward under the prefix relation. A \emph{path} through $T$ is a binary sequence $P \in 2^\omega$ such that every initial segment belongs to $T$.

A \emph{Turing ideal} $\Ical$ is a collection of sets which is closed downward under the Turing reduction and closed under the effective join, that is, $(\forall X \in \Ical)(\forall Y \leq_T X) Y \in \Ical$ and $(\forall X, Y \in \Ical) X \oplus Y \in \Ical$, where $X \oplus Y = \{ 2n : n \in X \} \cup \{ 2n+1 : n \in Y \}$. A \emph{Scott set} is a Turing ideal $\Ical$ such that every infinite binary tree $T \in \Ical$ has a path in $\Ical$. In other words, a Scott set is the second-order part of an $\omega$-model of $\rca + \wkl$.
A Turing ideal $\Mcal$ is \emph{countable coded} by a set $X$
if $\Mcal = \{ X_n : n \in \omega \}$ with $X = \bigoplus_n X_n$.
A formula is $\Sigma^0_1(\Mcal)$ (resp.\ $\Pi^0_1(\Mcal)$) if it is $\Sigma^0_1(X)$ (resp.\ $\Pi^0_1(X)$) for some $X \in \Mcal$.

Given two sets $A$ and $B$, we denote by $A < B$ the formula
$(\forall x \in A)(\forall y \in B)[x < y]$.
We write $A \subseteq^{*} B$ to mean that $A - B$ is finite, that is,
$(\exists n)(\forall a \in A)(a \not \in B \imp a < n)$.
A \emph{$k$-cover} of a set $X$ is a sequence of sets $Y_0, \dots, Y_{k-1}$ such that $X \subseteq Y_0 \cup \dots \cup Y_{k-1}$.

\section{Preliminary tools}
We start by introduce the central tools used in the various forcings to come : the largeness and partition regular classes. They were introduced by the authors in~\cite{Monin2018Pigeons} to design a notion of forcing controlling the second jump of solutions to the pigeonhole principle. In this paper we push their use further, with the introduction of $\Mcal$-cohesive and $\Mcal$-minimal largeness classes, which are necessary for the third jump control and beyond.

\subsection{Largeness classes}

\begin{definition}
A \emph{largeness class} is a non-empty collection of sets $\Acal \subseteq 2^\omega$ such that
\begin{itemize}
	\item[(a)] If $X \in \Acal$ and $Y \supseteq X$, then $Y \in \Acal$
	\item[(b)] For every $k$-cover $Y_0, \dots, Y_{k-1}$ of $\omega$, there is some $j < k$ such that $Y_j \in \Acal$.
\end{itemize}
\end{definition}

For example, the collection of all the infinite sets is a largeness class. Moreover, any superclass of a largeness class is again a largeness class. 

\begin{lemma}\label{lem:decreasing-largeness-yields-largeness}
Suppose $\Acal_0 \supseteq \Acal_1 \supseteq \dots$ is a decreasing sequence of largeness classes.
Then $\bigcap_s \Acal_s$ is a largeness class.
\end{lemma}
\begin{proof}
If $X \in \bigcap_s \Acal_s$ and $Y \supseteq X$, then for every $s$, since $\Acal_s$ is a largeness class, $Y \in \Acal_s$, so $Y \in \bigcap_s \Acal_s$.
Let $Y_0, \dots, Y_{k-1}$ be a $k$-cover of $\omega$. For every $s \in \omega$, there is some $j < k$
such that $Y_j \in \Acal_s$. By the infinite pigeonhole principle, there is some $j < k$ such that $Y_j \in \Acal_s$ for infinitely many $s$. Since $\Acal_0 \supseteq \Acal_1 \supseteq$ is a decreasing sequence,  $Y_j \in \bigcap_s \Acal_s$.
\end{proof}

\begin{lemma}\label{lem:largeness-class-complexity}
Let $\Acal$ be a $\Sigma^0_1$ class.
The sentence “$\Acal$ is a largeness class" is $\Pi^0_2$.
\end{lemma}
\begin{proof}
Say $\Acal = \{ X : (\exists \sigma \preceq X)\varphi(\sigma) \}$ where $\varphi$ is a $\Sigma^0_1$ formula.
By compactness, $\Acal$ is a largeness class iff for every $\sigma$ and $\tau$ such that $\sigma \subseteq \tau$ and $\varphi(\sigma)$ holds, $\varphi(\tau)$ holds, and for every $k$, there is some $n \in \omega$ such that for every $\sigma_0 \cup \dots \cup \sigma_{k-1} = \{0, \dots, n\}$, there is some $j < k$ such that $\varphi(\sigma_j)$ holds.
\end{proof}

\subsection{Partition regular classes}

\begin{definition}
A \emph{partition regular class} is a collection of sets $\Lcal \subseteq 2^\omega$ such that
\begin{itemize}
	\item[(a)] $\Lcal$ is a largeness class
	\item[(b)] For every $X \in \Lcal$ and $Y_0 \cup \dots \cup Y_{k-1} \supseteq X$, there is some $j < k$ such that $Y_j \in \Lcal$.
\end{itemize}
\end{definition}

In particular, the class of all infinite sets is partition regular.

\begin{lemma}\label{lem:decreasing-largeness-yields-pr}
Suppose $\Acal_0 \supseteq \Acal_1 \supseteq \dots$ is a decreasing sequence of partition regular classes.
Then $\bigcap_s \Acal_s$ is a partition regular class.
\end{lemma}
\begin{proof}
The proof is easy, similar to the one of \Cref{lem:decreasing-largeness-yields-largeness} and left to the reader.
\end{proof}

\begin{definition}
Let $\Acal$ be a largeness class. Define
$$
\Lcal(\Acal) = \{ X \in 2^\omega : \forall k\ \forall X_0 \cup \dots \cup X_{k-1} \supseteq X\ \exists i < k\ X_i \in \Acal \}
$$
\end{definition}

Note that a superset of a partition regular class need not to be partition regular, it is however always a largeness class. Note also that if $\Ucal$ is a $\Sigma^0_1(X)$ class, then by compactness $\Lcal(\Ucal)$ is a $\Pi^0_2(X)$ class.

\begin{lemma}\label{lem:lcal-of-largeness-is-partition-regular}
Let $\Acal$ be a largeness class. Then $\Lcal(\Acal)$ is the largest partition regular subclass of $\Acal$.
\end{lemma}
\begin{proof}
We first prove that $\Lcal(\Acal)$ is a partition regular subclass of $\Acal$.
By definition of $\Acal$ being a largeness class, $\omega \in \Lcal(\Acal)$.
Let $X \in \Lcal(\Acal)$ and $X_0 \cup \dots \cup X_{k-1} \supseteq X$. Suppose for the sake of absurd that $X_i \not \in \Lcal(\Acal)$ for every $i < k$. Then for every $i < k$, there is some $k_i \in \omega$ and some $Y^0_i \cup \dots \cup Y^{k_i-1}_i \supseteq X_i$ such that $Y^j_i \not \in \Acal$ for every $j < k_i$. Then $\{Y^j_i : i < k, j < k_i \}$ is a cover of $X$ contradicting $X \in \Lcal(\Acal)$. Therefore $\Lcal(\Acal)$ is a partition regular class.
Moreover, $\Lcal(\Acal) \subseteq \Acal$ as witnessed by taking the trivial cover of $X$ by $X$ itself.

We now prove that $\Lcal(\Acal)$ is the largest partition regular subclass of $\Acal$.
Indeed, let $\Bcal$ be a partition regular subclass of $\Acal$. Then for every $X \in \Bcal$, every $X_0 \cup \dots \cup X_{k-1} \supseteq X$, there is some $j < k$ such that $X_j \in \Bcal \subseteq \Acal$. Thus $X \in \Lcal(\Acal)$, so $\Bcal \subseteq \Lcal(\Acal)$.
\end{proof}

\subsection{$\Mcal$-cohesive classes}

We now introduce the notion of $\Mcal$-cohesive largeness classes for a countable Scott set $\Mcal$. One would ideally need $\Mcal$-minimal largeness classes instead for the upcoming forcing  (see \Cref{def_minmalclasses}). Unfortunately these classes are definitionally too complex for us. We use instead $\Mcal$-cohesive largeness classes, which are definitionally simpler and can be seen as a way to ``almost'' build a minimal largeness class. The key property of these classes lies in \Cref{lem:cohesive-largeness-compatibility}, which is later used to show that an $\Mcal$-cohesive largeness class contains a \emph{unique} $\Mcal$-minimal largeness class.

Given an infinite set $X$, we let $\Lcal_X$ be the $\Pi^0_2(X)$ largeness class of all sets having an infinite intersection with~$X$.

\begin{definition}
A class $\Acal$ is \emph{$\Mcal$-cohesive}
if for every $X \in \Mcal$, either $\Acal \subseteq \Lcal_X$ or $\Acal \subseteq \Lcal_{\overline{X}}$.
\end{definition}

In what follows, fix an effective enumeration $\Ucal^Z_0, \Ucal^Z_1, \dots$ of all the $\Sigma^{0,Z}_1$ classes upward-closed under the superset relation, that is, if $X \in \Ucal^Z_e$ and $Y \supseteq X$, then $Y \in \Ucal^Z_e$. Fix also a Scott set $\Mcal = \{X_0, X_1, \dots \}$ countable coded by a set $M$. Given a set $C \subseteq \omega^2$, we write
$$
\Ucal^\Mcal_C = \bigcap_{\langle e, i \rangle \in C} \Ucal^{X_i}_e
$$

\begin{lemma}\label{lem:cohesive-largeness-compatibility}
Let $\Ucal^{\Mcal}_C$ be an $\Mcal$-cohesive class.
Let $\Ucal^{\Mcal}_D$ and $\Vcal^{\Mcal}_E$ be such that $\Ucal^{\Mcal}_C \cap \Ucal^{\Mcal}_D$ and $\Ucal^{\Mcal}_C \cap \Ucal^{\Mcal}_E$ are both largeness classes. Then $\Ucal^{\Mcal}_C \cap \Ucal^{\Mcal}_D \cap \Ucal^{\Mcal}_E$ is a largeness class.
\end{lemma}
\begin{proof}
Suppose for contradiction that $\Ucal^{\Mcal}_C \cap \Ucal^{\Mcal}_D \cap \Ucal^{\Mcal}_E$ is not a largeness class. Then by Lemma~\ref{lem:decreasing-largeness-yields-largeness}, there is some finite $C_1 \subseteq C$, $D_1 \subseteq D$ and $E_1 \subseteq E$ such that  $\Ucal^{\Mcal}_{C_1} \cap \Ucal^{\Mcal}_{D_1} \cap \Ucal^{\Mcal}_{E_1}$ is not a largeness class. Since $\Ucal^{\Mcal}_{C_1} \cap \Ucal^{\Mcal}_{D_1} \cap \Ucal^{\Mcal}_{E_1}$ is $\Sigma^0_1(\Mcal)$, the collection $\Ccal$ of all sets $Y_0 \oplus \dots \oplus Y_{k-1}$ such that $Y_0 \sqcup \dots \sqcup Y_{k-1} = \omega$ and for every $i < k$, $Y_i \not \in \Ucal^{\Mcal}_{C_1} \cap \Ucal^{\Mcal}_{D_1} \cap \Ucal^{\Mcal}_{E_1} \supseteq \Ucal^{\Mcal}_{C} \cap \Ucal^{\Mcal}_{D} \cap \Ucal^{\Mcal}_{E}$, is a non-empty $\Pi^0_1(\Mcal)$ class. 
Since $\Mcal$ is a Scott set, $\Ccal \cap \Mcal \neq \emptyset$, so fix such a set $Y_0 \oplus \dots \oplus Y_{k-1} \in \Ccal \cap \Mcal$.
	Since $\Ucal^{\Mcal}_C$ is $\Mcal$-cohesive, there must be some $i < k$ such that $\Ucal^{\Mcal}_C \subseteq \Lcal_{Y_i}$. In particular, $Y_i \in \Ucal^{\Mcal}_C$, so $Y_i \not \in \Ucal^{\Mcal}_{D}$ or $Y_i \not \in \Ucal^{\Mcal}_{E}$. Suppose $Y_i \not \in \Ucal^{\Mcal}_{D}$, as the other case is symmetric. Since $Y_j \cap Y_i = \emptyset$ for every $j \neq i$, then $Y_j \not \in \Ucal^{\Mcal}_C \subseteq \Lcal_{Y_i}$ for every $j \neq i$. It follows that $Y_0, \dots, Y_{k-1}$ witnesses that $\Ucal^{\Mcal}_C \cap \Ucal^{\Mcal}_{D}$ is not a largeness class. Contradiction.
\end{proof}

\subsection{$\Mcal$-minimal classes}

\begin{definition} \label{def_minmalclasses}
A class $\Acal$ is \emph{$\Mcal$-minimal} if for every $X \in \Mcal$ and $e \in \omega$, either $\Acal \subseteq \Ucal^X_e$ or $\Acal \cap \Ucal^X_e$ is not a largeness class.
\end{definition}

The following is a corollary of \cref{lem:cohesive-largeness-compatibility} and informally says that an $\Mcal$-cohesive largeness class contains a unique $\Mcal$-minimal largeness class, which can be build with a greedy algorithm.

\begin{lemma}
Given an $\Mcal$-cohesive largeness class $\Ucal^{\Mcal}_C$,
the collection of sets
$$
\langle \Ucal^{\Mcal}_C \rangle = \bigcap_ {e \in \omega, X \in \Mcal} \{ \Ucal_e^X : \Ucal^{\Mcal}_C \cap \Ucal_e^X \mbox{ is a largeness class}\}
$$
is an $\Mcal$-minimal largeness class contained in $\Ucal^{\Mcal}_C$.
\end{lemma}
\begin{proof}
We first prove that $\langle \Ucal^{\Mcal}_C \rangle$ is a largeness class.
Let $e_0, e_1, \dots$ and $X_0, X_1, \dots$ be an enumeration of all pairs $(e, X) \in \omega \times \Mcal$ such that $\Ucal^{\Mcal}_C \cap \Ucal_e^X$ is a largeness class.
By induction on $n$ using Lemma~\ref{lem:cohesive-largeness-compatibility},
$\bigcap_{i < n} \Ucal_{e_i}^{X_i}$ is a largeness class for every $n \in \omega$. Thus, by Lemma~\ref{lem:decreasing-largeness-yields-largeness}, $\langle \Ucal^{\Mcal}_C \rangle = \bigcap_i \Ucal_{e_i}^{X_i}$ is a largeness class. By construction $\langle \Ucal^{\Mcal}_C \rangle$ is $\Mcal$-minimal.
\end{proof}

Note that we clearly have $\langle \Ucal^{\Mcal}_C \rangle \subseteq \Ucal^{\Mcal}_C$. The notation $\langle \Ucal^{\Mcal}_C \rangle$ for an $\Mcal$-cohesive largeness class will be used all along this document. Note that $\langle \Ucal^{\Mcal}_C \rangle = \Ucal^{\Mcal}_D$ where $D$ is the set of all $\langle e, i\rangle$ such that $\Ucal^{\Mcal}_C \cap \Ucal^{X_i}_e$ is a largeness class.

\begin{lemma}
Let $\Ucal^{\Mcal}_C$ be a largeness class. Then $\Lcal(\Ucal^{\Mcal}_C) = \Ucal^{\Mcal}_D$ for some $D \subseteq \omega^2$. Furthermore $D$ is computable from $C$.
\end{lemma}
\begin{proof}
Let $\Ucal^{\Mcal}_C$ be a largeness class. Note that $\Lcal(\Ucal^{\Mcal}_C) \subseteq \bigcap_{\langle e, i \rangle}\Lcal(\Ucal^{X_i}_e)$. By \cref{lem:decreasing-largeness-yields-pr} the class $\bigcap_{\langle e, i \rangle}\Lcal(\Ucal^{X_i}_e)$ is partition regular. By \cref{lem:lcal-of-largeness-is-partition-regular} we then must have $\Lcal(\Ucal^{\Mcal}_C) = \bigcap_{\langle e, i \rangle}\Lcal(\Ucal^{X_i}_e)$. Also we have by definition of $\Lcal(\Ucal)$ for a class $\Ucal$ that $\Lcal(\Ucal^{X_i}_e)$ is a $\Pi^0_2(X_i)$ class whose indices are computable uniformly in $e$.

Thus we have that $\Lcal(\Ucal^{\Mcal}_C) = \Ucal^{\Mcal}_D$ for some $D \subseteq \omega^2$. Furthermore $D$ is computable from $C$.
\end{proof}

\begin{corollary}\label{lem:minimal-is-partition-regular}
Suppose $\Ucal^{\Mcal}_C$ is an $\Mcal$-minimal largeness class. Then $\Ucal^{\Mcal}_C$ is partition regular.
\end{corollary}
\begin{proof}
Let $D$ be such that $\Ucal^{\Mcal}_D = \Lcal(\Ucal^{\Mcal}_C)$.
By Lemma~\ref{lem:lcal-of-largeness-is-partition-regular}, $\Ucal^{\Mcal}_D \subseteq \Ucal^{\Mcal}_C$. By $\Mcal$-minimality of $\Ucal^{\Mcal}_C$, $\Ucal^{\Mcal}_C \subseteq \Ucal^{\Mcal}_D$. It follows that $\Ucal^{\Mcal}_C = \Ucal^{\Mcal}_D$.  Since $\Ucal^{\Mcal}_D$ is partition regular, then so is $\Ucal^{\Mcal}_C$.
\end{proof}

It follows that if $\Ucal^{\Mcal}_C$ is an $\Mcal$-cohesive largeness class, then the $\Mcal$-minimal class $\langle \Ucal^{\Mcal}_C \rangle$ is a partition regular class.

\subsection{The framework}

We now build a sequence of sets $\{\Ucal^{\Mcal_n}_{C_n}\}_{n \in \omega}$ which will be used for the forcing in the next section.

\begin{proposition} \label{prop-hyp-scott}
There is a sequence of sets $\{M_n\}_{n < \omega}$ such that:
\begin{enumerate}
\item $M_n$ codes for a countable Scott set $\Mcal_n$
\item $\halt^{(n)}$ is uniformly coded by an element of $\Mcal_n$
\item Each $M_n'$ is uniformly computable in $\halt^{(n+1)}$
\end{enumerate}
\end{proposition}
\begin{proof}
Let us show the following: there is a functional $\Phi : 2^\omega \rightarrow 2^\omega$ such that for any oracle $X$, we have that $M' = \Phi(X')$ is such that $M = \oplus_{n \in \omega} X_n$ codes for a Scott set $\Mcal$ with $X_0 = X$.

Fix a uniformly computable enumeration $\Ccal^Y_0, \Ccal^Y_1, \dots$ of all non-empty $\Pi^0_1(Y)$ classes.
Let $\Dcal_X$ be the $\Pi^0_1(X)$ class of all $\bigoplus_n Y_n$ such that $Y_0 = X$ and for every $n = \langle a, b \rangle \in \omega$, $Y_{n+1} \in \Ccal_a^{\bigoplus_{j \leq b} Y_j}$. Note that this $\Pi^0_1(X)$ class is uniform in $X$ and any member of $\Dcal_X$ is a code of a Scott set whose first element is~$X$. Using the Low basis theorem~\cite{Jockusch197201}, there is a Turing functional $\Phi$ such that $\Phi(X')$ is the jump of a member of $\Dcal_X$ for any $X$.

Using this function $\Phi$, it is clear that uniformly in $\halt^{(n+1)}$ one can compute the jump of a set $M_n$ coding for a Scott set $\Mcal_n$ and containing $\halt^{(n)}$ as its first element.
\end{proof}

Let us assume that $\{\Mcal_n\}_{n < \omega}$ is a sequence which verifies \Cref{prop-hyp-scott}. Recall the notation $\langle \Ucal^\Mcal_C \rangle$ : the unique minimal largeness subclass of an $\Mcal$-cohesive largeness class.

\begin{proposition} \label{prop-hyp-cohesiveclassa}
There is a sequence of sets $\{C_n\}_{n \in \omega}$ such that:
\begin{enumerate}
\item $\Ucal_{C_n}^{\Mcal_{n}}$ is an $\Mcal_n$-cohesive largeness class
\item $\Ucal_{C_{n+1}}^{\Mcal_{n+1}} \subseteq \langle \Ucal_{C_n}^{\Mcal_{n}} \rangle$
\item Each $C_n$ is coded by an element of $\Mcal_{n + 1}$ uniformly in $n$ and $M_{n + 1}$.
\end{enumerate}
\end{proposition}

In order to prove \Cref{prop-hyp-cohesiveclassa} we use the two following uniformity lemmas, which will also be helpful later to continue the sequence of \Cref{prop-hyp-cohesiveclassa} through the computable ordinals (see \Cref{prop-hyp-cohesiveclass}).

\begin{lemma} \label{lem-hyp-cohesiveclass1}
There is a functional $\Phi : 2^\omega \times \omega \rightarrow 2^\omega$ such that for any set $M$ coding for a Scott set $\Mcal$, for any $e$ such that $C = \Phi_e(M'')$ is such that $\Ucal_C^{\Mcal}$ is an $\Mcal$-cohesive largeness class, $D = \Phi(M'', e)$ is such that $C \subseteq D$ and $\Ucal_D^{\Mcal} = \langle \Ucal_C^{\Mcal} \rangle$.
\end{lemma}
\begin{proof}
Say $\Mcal = \{X_0, X_1, \dots \}$ with $M = \bigoplus_i X_i$.
Let $\{\langle e_t, i_t \rangle\}_{t \in \omega}$ be an enumeration of $\omega \times \omega$. Suppose that at stage $t$ a finite set $D^t \subseteq \{\langle e_0, i_0 \rangle, \dots, \langle e_t, i_t \rangle\}$ has been defined such that $\Ucal_{D^t}^{\Mcal} \cap \Ucal_C^{\Mcal}$ is a largeness class and such that for any $s \leq t$, $\langle e_s, i_s \rangle \notin D^t$ implies that $\Ucal_{e_{s}}^{X_{i_{s}}} \cap \Ucal_{D^t}^{\Mcal} \cap \Ucal_C^{\Mcal}$ is not a largeness class.

Then at stage $t+1$, we ask $M''$ if $\Ucal_{e_{t+1}}^{X_{i_{t+1}}} \cap  \Ucal_{D^t}^{\Mcal} \cap \Ucal_C^{\Mcal}$ is a largeness class. If so we define $D^{t+1} = D^t \cup \{\langle e_{t+1}, i_{t+1} \rangle\}$. Otherwise we define $D^{t+1} = D^t$. Then $D = C \cup \bigcup_t D^t$ is uniformly $M''$-computable and $\Ucal_D^{\Mcal}$ equals $\langle \Ucal_C^{\Mcal} \rangle$.
\end{proof}

\begin{lemma} \label{lem-hyp-cohesiveclass2}
There is a functional $\Phi : 2^\omega \times \omega \times \omega \rightarrow \omega$ such that for any set $M$ coding for a Scott set $\Mcal$, for any set $N$ coding for a Scott set $\Ncal$ such that $M' \in \Ncal$ with $N$-index $i_M$, for any $C \in \Ncal$ with $N$-index $i_C$, such that $\Ucal_C^{\Mcal}$ is a partition regular class, $\Phi(N, i_M, i_C)$ is an $N$-index for $D \supseteq C$ such that $\Ucal_D^{\Mcal}$ is an $\Mcal$-cohesive largeness class.
\end{lemma}
\begin{proof}
The functional $\Phi$ does the following : It looks for $M'$ at index $i_M$ inside $\Ncal$. From $M'$ it computes $M = \oplus_{n} X_n$. It then computes with $M' + C$ the tree $T$ containing all the elements $\sigma$ such that
$$\left(\bigcap_{\sigma(i) = 0} 2^\omega - X_i\right) \cap \left(\bigcap_{\sigma(i) = 1} X_i \right) \in \bigcap_{\langle e, j \rangle \in C \upharpoonright {|\sigma|}} \Ucal_e^{X_j}$$

Clearly $[T]$ is not empty. The functional $\Phi$ then finds an $N$-index for an element $Y \in [T]$. For $\sigma \prec Y$ let $X_\sigma = (\bigcap_{\sigma(i) = 0} (2^\omega - X_i)) \cap (\bigcap_{\sigma(i) = 0} X_i)$. We must have for every $\sigma \prec Y$ that $X_\sigma \in \Ucal_C^{\Mcal}$. It follows as $\Ucal_C^{\Mcal}$ is partition regular, that for every $\sigma \prec Y$, $\L_{X_\sigma} \cap \Ucal_C^{\Mcal}$ is a largeness class. Thus $\bigcap_{\sigma \prec Y} \L_{X_\sigma} \cap \Ucal_C^{\Mcal}$ is an $\Mcal$-cohesive largeness class. Also $M \oplus Y \oplus C$ uniformly computes a set $D$ such that $\Ucal_D^{\Mcal} = \bigcap_{\sigma \prec Y} \L_{X_\sigma} \cap \Ucal_C^{\Mcal}$. The function $\Phi$ then returns an $N$-index for $D$.
\end{proof}

\begin{proof}[Proof of \Cref{prop-hyp-cohesiveclassa}]
Suppose that stage $n$ we have defined $C_n$ verifying $(1) (2)$ and $(3)$. Let us define $C_{n+1}$.

Note that the set $C_{n}$ is coded by an element of $\Mcal_{n + 1}$, and thus that $C_{n}$ is computable in $\halt^{(n+2)}$ and then computable in $M_n''$. Using \Cref{lem-hyp-cohesiveclass1} we define $D_n \supseteq C_n$ to be such that $\Ucal_{D_n}^{\Mcal_n} = \langle \Ucal_{C_n}^{\Mcal_n} \rangle$ and such that $D_n$ is uniformly $M_{n}''$-computable. We define $E_{n+1}$ to be the transfer of the $M_n$-indices constituting $D_n$ into $M_{n+1}$-indices, using that $M_{n}$ is an element of $M_{n+1}$. So we have $\Ucal_{E_{n+1}}^{\Mcal_{n+1}} = \Ucal_{D_n}^{\Mcal_{n}}$.

Note that as $E_{n+1}$ is computable in $M_{n}'' \oplus M_{n+1}$ and thus in $\halt^{((n+1)+1)}$. It is then coded by an element of $\Mcal_{(n+1)+1}$. Note also that $\Ucal_{E_{n+1}}^{\Mcal_{n+1}}$ is partition regular as it equals $\langle \Ucal_{C_n}^{\Mcal_n} \rangle$. Using \Cref{lem-hyp-cohesiveclass2} we uniformly find an $\Mcal_{(n+1)+1}$-index of $C_{n+1} \supseteq E_{n+1}$ to be such that $\Ucal_{C_{n+1}}^{\Mcal_{n+1}}$ is an $\Mcal_{n+1}$-cohesive largeness class.
\end{proof}

\section{Generalized Pigeonhole forcing}

The notion of forcing used to build solutions to the pigeonhole principle while controlling the first jump is a variant of Mathias forcing. In this section, we extend Mathias forcing to a more general notion of forcing while controlling iterated jumps, that is while tightly controlling the truth of $\Sigma^0_n$ and $\Pi^0_n$ formulas.

Let $\Mcal_0, \Mcal_1, \dots, \Mcal_n$ be countable Scott sets coded by sets $M_0, M_1, \dots, M_n$, respectively, satisfying (1)(2) and (3) of \cref{prop-hyp-scott}. Let $C_0, C_1, C_2$ be sequence of sets satisfying (1)(2) and (3) of \cref{prop-hyp-cohesiveclassa}, that is, $\Ucal^{\Mcal_n}_{C_n}$ is an $\Mcal$-cohesive largeness class, $\Ucal^{\Mcal_{n+1}}_{C_{n+1}} \subseteq \langle \Ucal^{\Mcal_n}_{C_n} \rangle$ and each $C_n$ is coded by an element of $\Mcal_{n+1}$.

\subsection{The forcing conditions}

\begin{definition}
For each $n \geq 0$ let $\Pb_n$ be the set of pairs $(\sigma, X)$ such that
\begin{itemize}
	\item[(a)] $X \cap \{ 0, \dots, |\sigma|\} = \emptyset$
	\item[(b)] $X \in \langle \Ucal^{\Mcal_{n}}_{C_{n}} \rangle$
\end{itemize}
\end{definition}

Note that $X$ is infinite for $(\sigma, X) \in \Pb_n$ since $\Ucal^{\Mcal_{n}}_{C_{n}}$ contains only infinite sets. Mathias forcing builds a single object $G$ by approximations (conditions) which consist in an initial segment $\sigma$ of $G$, and an infinite reservoir of integers. The purpose of the reservoir is to restrict the set of elements we are allowed to add to the initial segment. The reservoir therefore enriches the standard Cohen forcing by adding an infinitary negative restrain.

\begin{definition}
The partial order on $\Pb_n$ is defined by $(\tau, Y) \leq (\sigma, X)$
if $\sigma \preceq \tau$, $Y \subseteq X$ and $\tau - \sigma \subseteq X$.
\end{definition}

Given a collection $\Fcal \subseteq \Pb_n$, we let $G_\Fcal = \bigcup \{ \sigma : (\sigma, X) \in \Fcal \}$.

\subsection{The forcing question}

We now define what we call ``the forcing question'' : a relation between forcing conditions $p \in \Pb_n$ and $\Sigma^0_{m+1}$ formulas for $m \leq n$. The goal of the forcing question is to be definitionally not too complex, while being able to find extensions of conditions forcing formulas or their negation. The forcing question will also be used in the definition of the forcing relation, which is why it is introduced first.

\begin{definition} \label{def-hyp-forcingqua}
Let $\sigma \in 2^{<\omega}$. Let $(\exists x) \Phi_e(G, x)$ be a $\Sigma^0_1$ formula. Let $\sigma \qvdash (\exists x) \Phi(G, x)$ holds if
$$\{Y\ :\ (\exists \tau \subseteq Y - \{0, \dots, |\sigma|\})(\exists x) \Phi_e(\sigma \cup \tau, x)\} \cap \Ucal_{C_{0}}^{\Mcal_{0}}$$
is a largeness class. Then inductively, given a $\Sigma^0_{m+1}$ formula $(\exists x) \Phi_e(G, x)$ with free variable $x$ for $1 \leq m < \omega$, we let $\sigma \qvdash (\exists x) \Phi_e(G, x)$ holds if
$$\{Y\ :\ (\exists \tau \subseteq Y - \{0, \dots, |\sigma|\}) (\exists x) \sigma \cup \tau \nqvdash \neg \Phi_e(G, x)\} \cap \Ucal_{C_{m}}^{\Mcal_{m}}$$
is a largeness class.

For a condition $p = (\sigma, X) \in \Pb_{n}$ for some $n < \omega$ and a $\Sigma^0_{m+1}$ formula $(\exists x) \Phi_e(G, x)$ with free variable $x$ for some $m \leq n$, we write $p \qvdash (\exists x) \Phi_e(G, x)$ if $\sigma \qvdash (\exists x) \Phi_e(G, x)$.
\end{definition}

\begin{proposition} \label{prop-hyp-effectivalla}
Let $\sigma \in 2^{<\omega}$. Let $(\exists x)\Phi_e(G, x)$ be a $\Sigma^0_{m+1}$ formula  for $m \geq 0$
\begin{enumerate}
\item The set
$$\{Y\ :\ (\exists \tau \subseteq Y - \{0, \dots, |\sigma|\})(\exists x) \Phi_e(\sigma \cup \tau, x)\}$$
is an upward-closed $\Sigma^0_1$ open set if $m = 0$. The set 
$$\{Y\ :\ (\exists \tau \subseteq Y - \{0, \dots, |\sigma|\}) (\exists x) \sigma \cup \tau \nqvdash \neg \Phi_e(G, x)\}$$
is an upward-closed $\Sigma^0_1(C_{m-1} \oplus \halt^{(m)})$ open set if $m>0$.
\item The relation $\sigma \qvdash (\exists x) \Phi_e(G, x)$ is $\Pi^0_1(C_{m} \oplus \halt^{(m+1)})$.
\end{enumerate}
This is uniform in $\sigma$ and $e$.
\end{proposition}
\begin{proof}
This is done by induction on $m$. We start with $m = 0$. Let $(\exists x)\Phi_e(G, x)$ be a $\Sigma^0_{1}$ formula and $\sigma \in 2^{<\omega}$. It is clear that
$$\Ucal(e, \sigma) = \{Y\ :\ (\exists \tau \subseteq Y - \{0, \dots, |\sigma|\})(\exists x) \Phi_e(\sigma \cup \tau, x)\}$$
is an upward closed $\Sigma^0_1$ class. Then $\sigma \qvdash (\exists x) \Phi_e(G, x)$ iff $\Ucal(e, \sigma) \cap \Ucal_{C_{0}}^{\Mcal_{0}}$ is a largeness class, that is, iff for every finite set $F \subseteq C_{0}$, the class $\Ucal(e, \sigma) \cap \Ucal_{F}^{\Mcal_{0}}$ is a largeness class. By \Cref{lem:largeness-class-complexity}, for each $F \subseteq C_0$, the statement is $\Pi^0_2(M_0)$ uniformly in $F$, and thus $\Pi^0_1(M_0')$ uniformly in $F$. It is then $\Pi^0_1(\halt')$ uniformly in $F$. Thus the whole statement is $\Pi^0_1(C_0 \oplus \halt')$.

Suppose (1) and (2) are true for $m-1$, every $\Sigma^0_{m}$ formula and every $\sigma$. Let $\sigma \in 2^{<\omega}$ and let $(\exists x)\Phi_e(G, x)$ be a $\Sigma^0_{m+1}$ formula. Let
$$\Ucal(e, \sigma) = \{Y\ :\ (\exists \tau \subseteq Y - \{0, \dots, |\sigma|\}) (\exists x) \sigma \cup \tau \nqvdash \neg \Phi_e(G, x)\}$$

Let us show (1). For each $x \in \omega$, the formula $\neg \Phi_e(G, x)$ is $\Sigma^0_{m}$ uniformly in $x$ and in $e$. By induction hypothesis, the relation $\sigma \cup \tau \nqvdash \neg \Phi_e(G, x)$ is $\Sigma^0_1(C_{m-1} \oplus \halt^{(m)})$ uniformly in $\sigma \cup \tau$ in $x$ and in $e$. It follows that $\Ucal(e, \sigma)$ is an upward closed $\Sigma^0_1(C_{m-1} \oplus \halt^{(m)})$ class.

Let us now show (2). We have $\sigma \qvdash (\exists x) \Phi_e(G, x)$ iff $\Ucal(e, \sigma) \cap \Ucal_{C_{m}}^{\Mcal_{m}}$ is a largeness class. Also $\Ucal(e, \sigma) \cap \Ucal_{C_{m}}^{\Mcal_{m}}$ is a largeness class if for all $F \subseteq C_{m}$, the class $\Ucal(e, \sigma) \cap \Ucal_{F}^{\Mcal_{m}}$ is a largeness class. By \Cref{lem:largeness-class-complexity}, it is a $\Pi^0_2(M_{m})$ statement uniformly in $F$ and then a $\Pi^0_1(M_{m}')$ statement uniformly in $F$ and then a $\Pi^0_1(\halt^{(m+1)})$ statement uniformly in $F$. It follows that the statement \qt{$\Ucal(e, \sigma) \cap \Ucal_{C_{m}}^{\Mcal_{m}}$ is a largeness class} is $\Pi^0_1(C_{m} \oplus \halt^{(m+1)})$.
\end{proof}

\subsection{The forcing relation}

The relation $\qvdash$ is now used to define the forcing relation. 

\begin{definition} \label{def:qb2-forcing-relation}
Let $n \in \omega$. Let $p = (\sigma, X) \in \Pb_n$.
Let $(\exists x)\Phi_e(G, x)$ be a $\Sigma^0_1$ formula. We define
\begin{itemize}
	\item[(a)] $p \Vdash (\exists x)\Phi_e(G, x)$ if $(\exists x)\Phi_e(\sigma, x)$
	\item[(b)] $p \Vdash (\forall x)\Phi_e(G, x)$ if $(\forall \tau \subseteq X)(\forall x)\Phi_e(\sigma \cup \tau, x)$
\end{itemize}
Then inductively for $1 \leq m \leq n$. Let $(\exists x)\Phi_e(G, x)$ be a $\Sigma_{m+1}$ formula. We define
\begin{itemize}
	\item[(a)] $p \Vdash (\exists x) \Phi_e(G, x)$ if there is some $x \in \omega$ such that $p \Vdash \Phi_e(G, x)$

	\item[(b)] $p \Vdash (\forall x) \neg \Phi_e(G, x)$ if for every $\tau \subseteq X$ and every $x \in \omega$, $\sigma \cup \tau \qvdash \neg \Phi_e(G, x)$
\end{itemize}
\end{definition}

\begin{lemma} \label{lem-hyp-forcepia}
Fix $0 \leq m \leq n$. Let $p \in \Pb_{n}$. Let $(\exists x)\Phi_e(G, x)$ be a $\Sigma^0_{m+1}$ formula. Then $p \Vdash (\forall x) \neg \Phi_e(G, x)$ iff $q \qvdash \neg \Phi_e(G, x)$ for every $x \in \omega$ and every $q \leq p$.
\end{lemma}
\begin{proof}
Suppose $p \Vdash (\forall x) \neg \Phi_e(G, x)$ with $p = (\sigma, X)$. By definition of the forcing relation and forcing extensions it is clear that $q \qvdash \neg \Phi_e(G, x)$ for every $x$ and every $q \leq p$. Suppose now $q \qvdash \neg \Phi_e(G, x)$ for every $x$ and every $q \leq p$. Given any $\tau \subseteq X$ we have that $(\sigma \cup \tau, X - \{0, \dots, |\sigma \cup \tau|\})$ is a valid extension of $p$ for which we have $\sigma \cup \tau \qvdash \neg \Phi_e(G, x)$ for every $x$. It follows that $p \Vdash (\forall x) \neg \Phi_e(G, x)$.
\end{proof}

\begin{lemma}\label{lem:qb2-forcing-closed-under-extension}
Fix $0 \leq m \leq n$. Let $(\exists x)\Phi_e(G, x)$ be a $\Sigma^0_{m+1}$ formula. Let $p, q \in \Pb_n$ be such that $q \leq p$.
\begin{itemize}
	\item[(a)] If $p \Vdash (\exists x)\Phi_e(G, x)$ then so does $q$.
	\item[(b)] If $p \Vdash (\forall x)\neg \Phi_e(G, x)$ then so does $q$.
\end{itemize}
\end{lemma}
\begin{proof}
We proceed by induction on $m$. It is clear for $\Sigma^0_1$ formulas. For $m > 0$ let $(\exists x)\Phi_e(G, x)$ be a $\Sigma^0_{m+1}$ formula.

For (a), by definition, there is some $x \in \omega$ such that $p \Vdash \Phi_e(G, x)$. As $\Phi_e(G, x)$ is a $\Pi^0_{m}$ formula, by induction hypothesis, $q \Vdash \Phi_e(G, x)$ and thus $q \Vdash (\exists x)\Phi_e(G, x)$.

For (b), by \Cref{lem-hyp-forcepia}, for all $x \in \omega$ and all $r \leq p$,  $r \qvdash \neg \Phi_e(G, x)$. Thus if $q \leq p$, also for all $x$ and all $r \leq q$, $r \qvdash \neg \Phi_e(G, x)$. It follows still by \Cref{lem-hyp-forcepia} that $q \Vdash (\forall x)\neg \Phi_e(G, x)$.
\end{proof}


\subsection{The core lemmas}

We now show the core lemmas. The first one shows how to find extensions to force formulas, while the second one is the classic ``forcing imply truth'' whenever we work with generic enough filters.

\begin{lemma} \label{prop-hyp-forcexta}
Let $p \in \Pb_{n}$ with $p = (\sigma, X)$. Let $(\exists x)\Phi_e(G, x)$ be a $\Sigma^0_{m+1}$ formula for $0 \leq m \leq n$.
\begin{enumerate}
\item Suppose $p \qvdash (\exists x)\Phi_e(G, x)$. Then there exists $q \leq p$ with $q \in \Pb_n$ such that $q \Vdash (\exists x)\Phi(G, x)$.

\item Suppose $p \nqvdash (\exists x)\Phi_e(G, x)$. Then there exists $q \leq p$ with $q \in \Pb_n$ such that $q \Vdash (\forall x)\neg \Phi(G, x)$.
\end{enumerate}
\end{lemma}
\begin{proof}
Let $p \in \Pb_{n}$. We start with $m = 0$. Suppose $p \qvdash (\exists x) \Phi_e(G, x)$. Let
$$\Ucal(e, \sigma) = \{Y\ :\ (\exists \tau \subseteq Y - \{0, \dots, |\sigma|\})(\exists x) \Phi_e(\sigma \cup \tau, x)\}$$
The class $\Ucal(e, \sigma) \cap \Ucal_{C_0}^{\Mcal_0}$ is a largeness class. As $\Ucal_{C_0}^{\Mcal_0}$ is $\Mcal_0$-cohesive, then $\langle  \Ucal_{C_0}^{\Mcal_0} \rangle \subseteq \Ucal(e, \sigma)$. As $X \in \langle \Ucal_{C_{n}}^{\Mcal_{n}}\rangle \subseteq \langle \Ucal_{C_0}^{\Mcal_0} \rangle \subseteq \Ucal(e, \sigma)$, there is $\tau \subseteq X$ such that $(\exists x)\Phi_e(\sigma \cup \tau, x)$ holds. As $\langle \Ucal_{C_{n}}^{\Mcal_{n}}\rangle$ contains only infinite sets and is partition regular, $X - \{0, \dots, |\sigma \cup \tau|\} \in \langle \Ucal_{C_{n}}^{\Mcal_{n}}\rangle$. Then $(\sigma \cup \tau, X - \{0, \dots, \sigma \cup \tau\})$ is a valid extension of $(\sigma, X)$ such that $(\sigma \cup \tau, X - \{0, \dots, |\sigma \cup \tau|\}) \Vdash (\exists x) \Phi_e(G, x)$.

Suppose now $p \nqvdash (\exists x) \Phi_e(G, x)$. Then the class $\Ucal(e, \sigma) \cap \Ucal_{C_0}^{\Mcal_0}$ is not a largeness class. It follows that there is a finite set $F \subseteq C_0$ such that $\Ucal(e, \sigma) \cap \Ucal_{F}^{\Mcal_0}$ is not a largeness class. For $k$ let $\Pcal_k$ be the $\Pi^0_1(Z)$ class for some $Z \in \Mcal_0$ of covers $Y_0 \cup \dots \cup Y_{k} \supseteq \omega$ such that $Y_i \notin \Ucal(e, \sigma) \cap \Ucal_{F}^{\Mcal_0}$ for each $i \leq k$. As $\Ucal(e, \sigma) \cap \Ucal_{F}^{\Mcal_0}$ is not a largeness class there must be some $k$ such that $\Pcal_k$ is not empty. Then there are sets $Y_0 \oplus \dots \oplus Y_{k} \in \Mcal_0 \cap \Pcal_k$. As $\langle \Ucal_{C_{n}}^{\Mcal_{n}}\rangle$ is partition regular and as $X \in \langle \Ucal_{C_{n}}^{\Mcal_{n}}\rangle$ we have some $i \leq k$ such that $Y_i \cap X \in \langle \Ucal_{C_{n}}^{\Mcal_{n}}\rangle \subseteq \Ucal_{C_0}^{\Mcal_0}$. Thus $(\sigma, Y_i \cap X)$ is a valid extension of $(\sigma, X)$ for which $(\sigma, Y_i \cap X) \Vdash (\forall x) \neg \Phi(G, x)$.

Suppose now $m > 0$. Suppose $p \qvdash (\exists x) \Phi_e(G, x)$. Let
$$\Ucal(e, \sigma) = \{Y\ :\ (\exists \tau \subseteq Y - \{0, \dots, |\sigma|\}) (\exists x) \sigma \cup \tau \nqvdash \neg \Phi_e(G, x)\}$$
By definition, the class $\Ucal(e, \sigma) \cap \Ucal_{C_m}^{\Mcal_m}$ is a largeness class. As $\Ucal_{C_m}^{\Mcal_m}$ is $\Mcal_m$-cohesive and as, by \Cref{prop-hyp-effectivalla}, the set $\Ucal(e, \sigma)$ is a $\Sigma^0_1(Y)$ for some $Y \in \Mcal_m$, then $\langle \Ucal_{C_m}^{\Mcal_m} \rangle \subseteq \Ucal(e, \sigma)$. As $X \in \langle \Ucal_{C_{n}}^{\Mcal_{n}}\rangle \subseteq \langle \Ucal_{C_m}^{\Mcal_m} \rangle \subseteq \Ucal(e, \sigma)$, there is $\tau \subseteq X$ such that $\sigma \cup \tau \nqvdash \neg \Phi_e(G, x)$ for some $x$. Note that as $\langle \Ucal_{C_{n}}^{\Mcal_{n}}\rangle$ contains only infinite sets and is partition regular we have $X - \{0, \dots, |\sigma \cup \tau|\} \in \langle \Ucal_{C_{n}}^{\Mcal_{n}}\rangle$. Also $(\sigma \cup \tau, X - \{0, \dots, |\sigma \cup \tau|\})$ is a valid extension of $(\sigma, X)$ such that $(\sigma \cup \tau, X - \{0, \dots, |\sigma \cup \tau|\}) \nqvdash \neg \Phi_e(G, x)$. Now by induction hypothesis we have some $Y \in \Mcal_m$ with $(\sigma \cup \tau, X \cap Y) \leq (\sigma, X)$ and such that $(\sigma \cup \tau, X \cap Y) \Vdash \Phi_e(G, x)$. It follows that $(\sigma \cup \tau, X \cap Y) \Vdash (\exists x)\Phi_e(G, x)$.

Suppose now $p \nqvdash (\exists x)\Phi_e(G, x)$. Then $\Ucal(e, \sigma) \cap \Ucal_{C_m}^{\Mcal_m}$ is not a largeness class. It follows that there is a finite set $F \subseteq C_0$ such that $\Ucal(e, \sigma) \cap \Ucal_{F}^{\Mcal_m}$ is not a largeness class. For $k$ let $\Pcal_k$ be the $\Pi^0_1(Z)$ class for some $Z \in \Mcal_m$ of covers $Y_0 \cup \dots \cup Y_{k} \supseteq \omega$ such that $Y_i \notin \Ucal(e, \sigma) \cap \Ucal_{F}^{\Mcal_m}$ for each $i \leq k$. As $\Ucal(e, \sigma) \cap \Ucal_{F}^{\Mcal_m}$ is not a largeness class there must be some $k$ such that $\Pcal_k$ is not empty. There are sets $Y_0 \oplus \dots \oplus Y_{k} \in \Mcal_m \cap \Pcal_k$. As $\langle \Ucal_{C_{n}}^{\Mcal_{n}}\rangle$ is partition regular and as $X \in \langle \Ucal_{C_{n}}^{\Mcal_{n}}\rangle$, there is some $i \leq k$ such that $Y_i \cap X \in \langle \Ucal_{C_{n}}^{\Mcal_{n}}\rangle \subseteq \Ucal_{C_m}^{\Mcal_m}$. It follows that $Y_i \cap X \notin \Ucal(e, \sigma)$. It means that for every $\tau \subseteq Y_i \cap X$ and every $x \in \omega$, $\sigma \cup \tau \qvdash \neg\Phi_e(G, x)$. It follows that $(\sigma, Y_i \cap X) \Vdash (\forall x)\neg\Phi_e(G, x)$.
\end{proof}

We now sow that forcing implies truth. We define first for that the precise level of genericity that we need.

\begin{definition} \label{def:genericenough}
Let $\Fcal \subseteq \Pb_n$ be a filter. The set $\Fcal$ is $m$-generic if for every $k \leq m$ and every $\Sigma^0_{k+1}$ formula $(\exists x) \Phi_e(G, x)$ there is a condition $p \in \Fcal$ such that $p \Vdash (\exists x) \Phi_e(G, x)$ or $p \Vdash (\forall x) \neg \Phi_e(G, x)$.
\end{definition}

Note that if a filter is $n$-generic, then it is $m$-generic for every  $m < n$.

\begin{lemma} \label{lem:forcing-implytruth}
Let $\Fcal \subseteq \Pb_{n}$ be an $n-1$-generic filter. Let $p \in \Fcal$. Let $(\exists x)\Phi_e(G, x)$ be a $\Sigma^0_{m+1}$ class for $0 \leq m \leq n$. 
\begin{itemize}
\item[(a)] Suppose $p \Vdash (\exists x) \Phi_e(G, x)$. Then $(\exists x) \Phi_e(G_\Fcal, x)$ holds. 
\item[(b)] Suppose $p \Vdash (\forall x) \neg \Phi_e(G, x)$. Then $(\forall x) \neg \Phi_e(G_\Fcal, x)$ holds.
\end{itemize}
\end{lemma}
\begin{proof}
The proof is done by induction on $m$. Let $p \in \Pb_{n}$ with $p = (\sigma, X)$. The result is clear and well-known for $m=0$. Suppose now $m>0$ and let $(\exists x)\Phi_e(G, x)$ be a $\Sigma^0_{m+1}$ formula. Suppose $p \Vdash (\exists x) \Phi_e(G, x)$. Then there exists $x$ such that $p \Vdash \Phi(G, x)$. By induction hypothesis $\Phi(G_\Fcal, x)$ hods and then $\exists x\ \Phi(G_\Fcal, x)$ holds.

Suppose $p \Vdash (\forall x) \neg \Phi_e(G, x)$. Then by \Cref{lem-hyp-forcepia} for every $x$ and every $q \leq p$, $q \qvdash \neg \Phi_e(G, x)$. From \Cref{prop-hyp-forcexta}, for every $x \in \omega$ and every $q \leq p$, there is some $r \leq q$ such that $r \Vdash \neg \Phi_e(G, x)$. It follows that for every $x$, the set $\{r \in \Pb_n\ :\ r \Vdash \neg \Phi_e(G, x)\}$ is dense below $p$. As $\Fcal$ is $n-1$-generic and $p \in \Fcal$ there must be for every $x$ some $q \in \Fcal$ such that $q \Vdash \neg \Phi_{e}(G, x)$. By induction hypothesis $\neg \Phi_e(G_{\Fcal}, x)$ for every $x$ and then $(\forall x) \neg \Phi_e(G_{\Fcal}, x)$ holds.
\end{proof}

\section{Cone avoidance under $\Delta^0_n$ reductions}

We show in this section the first and third theorems of the introduction --- \Cref{maintheorem1} and \Cref{maintheorem3}. The proof of \Cref{maintheorem2} will be postponed to the next section, where it will be achieved together with hyperarithmetic cone avoidance. We fix  a set $A^0 \sqcup A^1 = \omega$. We sometimes write $A$ for $A^0$ (with then $\omega - A = A^1$).

Unfortunately the above forcing is definitionally a bit too complex : the forcing question for $\Sigma^0_{m+1}$ statements is $\Pi^0_{m+2}$, whereas we would need it to be $\Sigma^0_{m+1}$.

For this reason, we need to plug upon the previous forcing another forcing notion, used only for ``the last step'' in formula induction. The drawbacks of this other forcing notions is that we are compelled to build two generic objects : one inside $A^0$ and one inside $A^1$. We then used the pairing argument first designed by Dzhafarov and Jockusch \cite{Dzhafarov2009Ramseys} to show that one of the object we build is sufficiently generic in the sense of \Cref{def:genericenough}.

\subsection{Another forcing on the top}

%

\begin{definition}\label{def:pb2-forcing-conditions}
Fix $n \geq 0$. Let $\Qb_n$ denote the set of conditions $(\sigma^0, \sigma^1, X)$ such that
\begin{itemize}
	\item[(a)] $\sigma^i \subseteq A^i$ for every $i < 2$
	\item[(b)] $X \cap \{ 0, \dots, \max_i |\sigma^i|\} = \emptyset$
	\item[(c)] $X \in \Mcal_n$
	\item[(d)] $X$ is infinite if $n = 0$ and $X \in \langle \Ucal^{\Mcal_{n-1}}_{C_{n-1}} \rangle$ if $n \geq 1$.
\end{itemize}
A forcing condition $(\sigma^0, \sigma^1, X) \in \Qb_n$ is \emph{valid for side} $i$ if $X \cap A^i \in \langle \Ucal^{\Mcal_{n-1}}_{C_{n-1}} \rangle$ for $n>0$ and if $X \cap A^i$ is infinite for $n = 0$.
\end{definition}

By definition of a Turing ideal $\Mcal$ countable coded by a set $M$, then $\Mcal$ can be written as $\{Z_0, Z_1, \dots \}$ with $M = \bigoplus_i Z_i$.
We then say that $i$ is an \emph{$M$-index} of $Z_i$. Thanks to the notion of index, any $\Qb_n$-condition can be finitely presented as follows. An \emph{index} of a $\Qb_n$-condition $c = (\sigma^0, \sigma^1, X)$ is a tuple $(\sigma^0, \sigma^1, a)$ where $a$ is an $M_n$-index for $X$.

\begin{definition}
The partial order on $\Qb_n$ is defined by
$$
(\tau^0, \tau^1, Y) \leq (\sigma^0, \sigma^1, X)
$$
if for every $i < 2$, $(\tau^i, Y) \leq (\sigma^i, X)$.
\end{definition}

Given a condition $c = (\sigma^0, \sigma^1, X)$ and $i < 2$, we write $c^{[i]} = (\sigma^i, X)$.
Each $\Qb_n$-condition $c$ represents two $\Pb_{n-1}$-conditions $c^{[0]}$ and $c^{[1]}$.


We now design a disjunctive forcing question which builds upon the forcing question of $\Pb_n$ conditions. The difference is that it is only used at the last step of the induction of formulas.

\begin{definition}
Let $c = (\sigma^0, \sigma^1, X) \in \Qb_0$ and let $(\exists x) \Phi_{e_0}(G, x)$ and $(\exists x) \Phi_{e_1}(G, x)$ be two $\Sigma_1$ formulas. Define the relation
$$
c \qvdash (\exists x)\Phi_{e_0}(G^0, x) \vee (\exists x)\Phi_{e_1}(G^1, x)
$$
to hold if for every 2-cover $Z^0 \cup Z^1 = X$, there is some side $i < 2$, some finite set $\rho \subseteq Z^i$ and some $x \in \omega$ such that $\Phi_{e_i}(\sigma^i \cup \rho, x)$ holds.

Let $n>0$. Let $c = (\sigma^0, \sigma^1, X) \in \Qb_n$ and let $(\exists x) \Phi_{e_0}(G, x)$ and $(\exists x) \Phi_{e_1}(G, x)$ be two $\Sigma^0_{n+1}$ formulas. Define the relation
$$
c \qvdash (\exists x)\Phi_{e_0}(G^0, x) \vee (\exists x)\Phi_{e_1}(G^1, x)
$$
to hold if for every 2-cover $Z^0 \cup Z^1 = X$, there is some side $i < 2$, some finite set $\rho \subseteq Z^i$ and some $x \in \omega$ such that $\sigma^i \cup \rho \nqvdash \neg \Phi_{e_i}(G, x)$ holds.
\end{definition}

\subsection{The complexity aspects of the $\Qb_n$ forcing}

This new forcing question now has the right definitional complexity

\begin{lemma}\label{lem:pb2-forcing-question-complexity}
Let $n \in \omega$. Let $c \in \Qb_n$ and let $(\exists x) \Phi_{e_0}(G, x)$ and $(\exists x) \Phi_{e_1}(G, x)$ be two $\Sigma^0_{n+1}$ formulas. The relation
$$
c \qvdash (\exists x)\Phi_{e_0}(G^0, x) \vee (\exists x)\Phi_{e_1}(G^1, x)
$$
is $\Sigma^0_1(Y)$ for some $Y \in \Mcal_n$. Moreover an $M_n$-index for $Y$ can be found uniformly in an index for $c$.
\end{lemma}
\begin{proof}
By compactness, for $n = 0$ the relation holds if there is a finite set $E \subseteq X$ such that for every $E_0 \cup E_1 = E$, there is some $i < 2$, some $\rho \subseteq E_i$ and $x_n \in \omega$ such that $\Phi_{e_i}(\sigma^i \cup \rho, x)$ holds, which is a $\Sigma^0_1(X)$ event for $X \in \Mcal_0$.

For $n > 0$ the relation holds if there is a finite set $E \subseteq X$ such that for every $E_0 \cup E_1 = E$, there is some $i < 2$, some $\rho \subseteq E_i$ and $x_n \in \omega$ such that $\sigma^i \cup \rho \nqvdash \neg \Phi_{e_i}(G, x)$ holds.
By \Cref{prop-hyp-effectivalla}, this statement is $\Sigma^0_1(X \oplus C_{n-1} \oplus \halt^{(n)})$ and then $\Sigma^0_1(Y)$ for some $Y \in \Mcal_n$.
\end{proof}

Before we continue, we need to study the effectivness of \Cref{prop-hyp-forcexta} about the forcing question for the $\Pb_n$ forcing.

\begin{lemma} \label{prop-hyp-forcexta-effect}
Let $n>0$. Let $c \in \Qb_{n}$ with $c = (\sigma_0, \sigma_1, X)$. Let $(\exists x)\Phi_e(G, x)$ be a $\Sigma^0_{m+1}$ formula for $0 \leq m < n$. Let $p = c^{[i]}$ for some $i < 2$ with $p = (\sigma, X)$.
\begin{enumerate}
\item Suppose $p \qvdash (\exists x)\Phi_e(G, x)$. The forcing condition $q \leq p$ of \Cref{prop-hyp-forcexta} which forces $(\exists x)\Phi_e(G, x)$ can always be of the form $(\sigma \cup \tau, X \cap Y)$ for $Y \in \Mcal_m$ where $\tau$ and an $M_m$-index for $Y$ can be found uniformly in any PA over $\halt^{(n+1)}$. If furthermore $c$ is valid on side $i$ one can ensure $\tau \subseteq A^i \cap X$ uniformly in $A \oplus P$ for any $P$ which is PA over $\halt^{(n+1)}$.

\item Suppose $p \nqvdash (\exists x)\Phi_e(G, x)$. The forcing condition $q \leq p$ of \Cref{prop-hyp-forcexta} which forces $(\forall x) \neg \Phi_e(G, x)$ can always be of the form $(\sigma, X \cap Y)$ for $Y \in \Mcal_m$ where an $M_m$ index for $Y$ can be found uniformly in any PA over $\halt^{(n+1)}$.
\end{enumerate}
\end{lemma}
\begin{proof}
Suppose $p \qvdash (\exists x) \Phi_e(G, x)$. By the proof of \Cref{prop-hyp-forcexta} there is $\tau \subseteq X$ such that $(\exists x)\Phi_e(\sigma \cup \tau, x)$ holds if $m=0$ and such that $\sigma \cup \tau \nqvdash \neg \Phi_e(G, x)$ for some $x$ if $m>0$. Note that if $X \cap A^i \in \langle \Ucal_{C_{n}}^{\Mcal_{n}}\rangle$, still refering to the proof of \Cref{prop-hyp-forcexta} we can ensure $\tau \subseteq X \cap A^i$. Also finding $\tau$ is a $\Sigma^0_1(X)$ event if $m=0$ and a $\Sigma^0_1(X \oplus C_{m-1} \oplus \halt^{m})$ if $m>0$ (resp. a $\Sigma^0_1(X\oplus A)$ if $m=0$ and a $\Sigma^0_1(A \oplus X \oplus C_{m-1} \oplus \halt^{m})$ if $m>0$). As $X \in \Mcal_{n}$ we can then find $\tau$ uniformly in $\halt^{n+1}$ (resp. in $\halt^{n+1} \oplus A$).

Suppose now $p \nqvdash (\exists x)\Phi_e(G, x)$. By the proof of \Cref{prop-hyp-forcexta} there is a finite set $F \subseteq C_{m}$ such that $\Ucal(e, \sigma) \cap \Ucal_{F}^{\Mcal_0}$ is not a largeness class. Note that finding $F$ is a $\Sigma^0_1(\halt^{(m+2)})$ event. It can then be found uniformly in $\halt^{(m+2)}$ and then uniformly in $\halt^{(n+1)}$. Still by the proof of \Cref{prop-hyp-forcexta} there must be some $k$ such that $\Pcal_k$ is not empty where $\Pcal_k$ is the $\Pi^0_1(Z)$ class for some $Z \in \Mcal_m$ of covers $Y_0 \cup \dots \cup Y_{k} \supseteq \omega$ such that $Y_i \notin \Ucal(e, \sigma) \cap \Ucal_{F}^{\Mcal_m}$ for each $i \leq k$. Searching for the first such $k$ is a $\Sigma^0_1(\halt^{m+1})$ event. Once found, one also compute uniformly in $M_m$ an index for $Y_0 \oplus \dots \oplus Y_{k} \in \Mcal_m \cap \Pcal_k$. As $\langle \Ucal_{C_{n-1}}^{\Mcal_{n-1}}\rangle$ is partition regular and as $X \in \langle \Ucal_{C_{n-1}}^{\Mcal_{n-1}}\rangle$, there is some $i \leq k$ such that $Y_i \cap X \in \langle \Ucal_{C_{n-1}}^{\Mcal_{n-1}}\rangle \subseteq \Ucal_{C_m}^{\Mcal_m}$. Finding the right $Y_i$ for $i \leq k$ is a $\Pi^0_1(C_{n-1} \oplus (Y_i \cap X \oplus M_{n-1})')$ event. As $X \in \Mcal_{n}$ it can then be found in any PA over $\halt^{(n+1)}$.
\end{proof}

We shall now show the extension of \Cref{prop-hyp-forcexta} for the $\Qb_n$ forcing conditions.

\begin{lemma}\label{lem:pb2-forcing-question-spec}
Let $n \in \omega$. Let $c \in \Qb_n$ and let $(\exists x)\Phi_{e_0}(G, x)$ and $(\exists x)\Phi_{e_1}(G, x)$ be two $\Sigma^0_{n+1}$ formulas.
\begin{itemize}
	\item[(a)] If $c \qvdash (\exists x)\Phi_{e_0}(G^0, x) \vee (\exists x)\Phi_{e_1}(G^1, x)$, then there is some $d \leq c$ and some $i < 2$ such that $$d^{[i]} \Vdash (\exists x)\Phi_{e_i}(G^i, x)$$
	\item[(b)] If $c \nqvdash (\exists x)\Phi_{e_0}(G^0, x) \vee (\exists x)\Phi_{e_1}(G^1, x)$, then there is some $d \leq c$ and some $i < 2$ such that $$d^{[i]} \Vdash (\forall x)\neg \Phi_{e_i}(G^i, x)$$
\end{itemize}
Moreover an index of $d$ can be found in $A \oplus P$ for any set $P$ which is PA over $\halt^{(n+1)}$ uniformly in an index of $c$, $e_0$ and $e_1$.
\end{lemma}
\begin{proof}
Say $c = (\sigma^0, \sigma^1, X)$. Both (a) and (b) are trivial in the case $n=0$. We treat the case $n > 0$.

(a) Let $Z^0 = X \cap A^0$ and $Z^1 = X \cap A^1$. Unfolding the definition of the forcing question, there is some $i < 2$, some $\rho \subseteq Z^i$ and $x \in \omega$ such that $\sigma_i \cup \rho \nqvdash \neg \Phi_{e_i}(G^i, x)$. By \Cref{prop-hyp-forcexta-effect} we have a set $Y \in \Mcal_{n}$ such that $(\sigma_i \cup \rho, X \cap Y) \leq (\sigma_i \cup \rho, X)$ and $(\sigma_i \cup \rho, X \cap Y) \Vdash (\exists x)\Phi_{e_i}(G^i, x)$. Note that $d = (\sigma_i \cup \rho, \sigma_{i-1}, X \cap Y)$ is a valid extension of $c$. From \Cref{prop-hyp-effectivalla} finding $\rho$ is a $\Sigma^0_1(A \oplus X \oplus C_{n-1} \oplus \halt^{(n)})$ event. From \Cref{prop-hyp-forcexta-effect} one can then find and $M_n$-index of $Y$ in any set $P$ which is PA over $\halt^{(n+1)}$. Overall an index for $d$ can be found in $A \oplus P$ for any set $P$ which is PA over $\halt^{(n+1)}$, uniformly in an index of $c$, $e_0$ and $e_1$.

(b) Let $\Dcal$ be the $\Pi^0_1(\Mcal_n)$ class of all $Z^0 \oplus Z^1$ with $Z^0 \cup Z^1 = X$, such that for every $i < 2$, every $\rho \subseteq Z^i$, and every $x \in \omega$ we have $\sigma \cup \rho \qvdash \neg \Phi_{e_i}(G^i, x)$. Let $Z^0 \oplus Z^1 \in \Dcal$ such that $Z^0 \oplus Z^1 \in \Mcal_n$. Since $\langle \Ucal^{\Mcal_{n-1}}_{C_{n-1}} \rangle$ is a partition regular class containing $X$, there is some $i < 2$ such that $Z^i \in \langle \Ucal^{\Mcal_{n-1}}_{C_{n-1}} \rangle$. Define the $\Qb_n$-condition $d = (\sigma^0, \sigma^1, Z^i)$. Then $d^{[i]} \Vdash (\forall x)\neg \Phi_{e_i}(G, x)$. Finding the right $Z_i$ is a $\Pi^0_1(C_{n-1} \oplus (X \oplus Z_i \oplus M_{n-1})')$ event. It can the be found uniformly in any set $P$ which is PA over $\halt^{(n+1)}$. This completes the proof of the lemma.
\end{proof}

\subsection{The degenerate forcing question}

The forcing question will be used with a disjunctive argument. Doing so we will build two generics, one in $A^0$ and one in $A^1$. Possibly only one of them will force every $\Sigma^0_n$ statement or their negation. The challenge is to ensure in the same time that the same generic also forces every $\Sigma^0_m$ statement or their negation for $m < n$, so that we can then apply \Cref{lem:forcing-implytruth} saying that forcing implies truth. It is only possible to do so on side $i$ under the assumption that our current forcing condition is valid on side $i$:

\begin{lemma}
Let $n \geq 0$. Let $c \in \Qb_n$ be valid for side $i$. Let $m < n$ and let $(\exists x) \Phi_e(G, x)$ be a $\Sigma^0_{m+1}$ formula. Then one can find uniformly in $A \oplus P$ for any $P$ which is PA over $\halt^{(n+1)}$, a condition $d \leq c$ such that $d^{[i]} \Vdash (\exists x) \Phi_e(G, x)$ or $d^{[i]} \Vdash (\forall x) \neg \Phi_e(G, x)$
\end{lemma}
\begin{proof}
We ask if $(\sigma_i, X) \qvdash (\exists x)\Phi_{e}(G, x)$. From \Cref{prop-hyp-forcexta-effect} if the answer is yes there is a $\tau \subseteq X \cap A^i$ and a set $Y \in \Mcal_n$ such that $d = (\sigma_i \cup \tau, Y \cap X \cap A^i) \qvdash (\exists x) \Phi_{e}(G, x)$. If no then there is $Y \in \Mcal_n$ such that $d = (\sigma_i \cup \tau, Y \cap X \cap A^i) \qvdash (\forall x) \neg \Phi_{e}(G, x)$.

In any case from \Cref{prop-hyp-forcexta-effect} an index for $d$ can be found uniformly in $A \oplus P$ for any $P$ which is PA over $\halt^{(n+1)}$.
\end{proof}

The difficulty is now to make sure that the side $i$ which turns out to be the right one, is also always a valid one. To do so we need a ``degenerate forcing question''.

\begin{definition}
Let $n > 0$. Let $c \in \Qb_n$. Let $\Ucal$ be $\Sigma^0_1(\halt^{(n)})$ large open set. Let $(\exists x)\Phi_{e}(G, x)$ be a $\Sigma^0_{n+1}$ formula. We define 
$$
c \qvdash^\Ucal (\exists x)\Phi_{e}(G, x)
$$
to hold if for every $Z^0 \cup Z^1 = X$, there exists $i<2$ such that $Z^i \in \Ucal$ and such that there is some $\rho \subseteq Z^i$ and $x_n \in \omega$ for which $\sigma^i \cup \rho \nqvdash \neg \Phi_{e}(G, x)$ holds.
\end{definition}

\begin{lemma} \label{noideaeffploufplouf}
Let $n > 0$. Let $(\sigma_0, \sigma_1, X) \in \Qb_n$. Let $\Ucal \supseteq \langle \Ucal^{\Mcal_{n-1}}_{C_{n-1}} \rangle$ be $\Sigma^0_1(\halt^{(n)})$ large open set such that $X \cap A^{i-1} \notin \Ucal$. Let $(\exists x)\Phi_e(G,x)$ be a $\Sigma^0_{n+1}$ formula. The statement 
$$
c \qvdash^\Ucal (\exists x)\Phi_{e}(G, x)
$$
is $\Sigma^0_1(Y)$ for some $Y \in \Mcal_n$.  Moreover an $M_n$-index for $Y$ can be found uniformly in an index for $c$.
\end{lemma}
\begin{proof}
The relation holds if there is a finite set $E \subseteq X$ such that for every $E_0 \cup E_1 = E$, there is some $i < 2$, some $\rho \subseteq E_i$ and $x_n \in \omega$ such that $[E_i] \subseteq \Ucal$ and $\sigma^i \cup \rho \nqvdash \neg \Phi_{e_i}(G, x)$ holds.
By \Cref{prop-hyp-effectivalla}, this statement is $\Sigma^0_1(X \oplus C_{n-1} \oplus \halt^{(n)})$ and then $\Sigma^0_1(Y)$ for some $Y \in \Mcal_n$.
\end{proof}

\begin{lemma}  \label{noideaeffploufplouf2}
Let $n > 0$. Let $(\sigma_0, \sigma_1, X) \in \Qb_n$. Let $\Ucal \supseteq \langle \Ucal^{\Mcal_{n-1}}_{C_{n-1}} \rangle$ be $\Sigma^0_1(\halt^{(n)})$ large open set such that $X \cap A^{i-1} \notin \Ucal$. Let $(\exists x)\Phi_e(G,x)$ be a $\Sigma^0_{n+1}$ formula.
\begin{itemize}
\item[(a)] Suppose $(\sigma_0, \sigma_1, X) \qvdash^{\Ucal} (\exists x)\Phi_{e}(G, x)$. 

\noindent Then there exists $d \leq c$ such that $d^{[i]} \Vdash (\exists x)\Phi_{e}(G, x)$
\item[(b)] Suppose $(\sigma_0, \sigma_1, X) \nqvdash^{\Ucal} (\exists x)\Phi_{e}(G, x)$. 

\noindent Then there exists $d \leq c$ such that $d^{[i]} \Vdash (\forall x) \neg \Phi_{e}(G, x)$
\end{itemize}
Furthermore an index for $d$ can be found in $A \oplus P$ for any set $P$ which is PA over $\halt^{(n+1)}$, uniformly in an index for $c$.
\end{lemma}
\begin{proof}
Say $c = (\sigma^0, \sigma^1, X)$. 

(a) Let $Z^0 = X \cap A^0$ and $Z^1 = X \cap A^1$. Unfolding the definition of the forcing question, there is some $j < 2$ such that $Z_j \in \Ucal$ and such that for some $\rho \subseteq Z^j$ and $x \in \omega$ we have $\sigma_j \cup \rho \nqvdash \neg \Phi_{e_i}(G^i, x)$. By hypothesis $Z^{i-1} \notin \Ucal$. Thus $i=j$ and by \Cref{prop-hyp-forcexta-effect} we have a set $Y \in \Mcal_{n}$ such that $(\sigma_i \cup \rho, X \cap Y) \leq (\sigma_i \cup \rho, X)$ and $(\sigma_i \cup \rho, X \cap Y) \Vdash (\exists x)\Phi_{e_i}(G^i, x)$. Note that $(\sigma_i \cup \rho, \sigma_{i-1}, X \cap Y)$ is a valid extension of $(\sigma_0, \sigma_1, X)$. From \Cref{prop-hyp-effectivalla} finding $\rho$ is a $\Sigma^0_1(A \oplus X \oplus C_{n-1} \oplus \halt^{(n)})$ event. From \Cref{prop-hyp-effectivalla} finding $\rho$ is a $\Sigma^0_1(A \oplus X \oplus C_{n-1} \oplus \halt^{(n)})$ event. From \Cref{prop-hyp-forcexta-effect} one can then find and $M_n$-index of $Y$ in any set $P$ which is PA over $\halt^{(n+1)}$. Overall an index for $d$ can be found in $A \oplus P$ for any set $P$ which is PA over $\halt^{(n+1)}$, uniformly in an index of $c$, $e_0$ and $e_1$.

(b) Let $\Dcal$ be the $\Pi^0_1(\Mcal_n)$ class of all $Z^0 \oplus Z^1$ with $Z^0 \cup Z^1 = X$, such that for every $i < 2$, $Z^i \notin \Ucal$ or for every $\rho \subseteq Z^i$, and every $x \in \omega$ we have $\sigma \cup \rho \qvdash \neg \Phi_{e}(G, x)$. Let $Z^0 \oplus Z^1 \in \Dcal$ be such that $Z^0 \oplus Z^1 \in \Mcal_n$. Since $\langle \Ucal^{\Mcal_{n-1}}_{C_{n-1}} \rangle$ is a partition regular class containing $X$, there is some $i < 2$ such that $Z^i \in \langle \Ucal^{\Mcal_{n-1}}_{C_{n-1}} \rangle$. Since $\Ucal \supseteq \langle \Ucal^{\Mcal_{n-1}}_{C_{n-1}} \rangle$ we must have $Z^i \in \Ucal$ and thus $d = (\sigma^0, \sigma^1, Z^i)$ is a $\Qb_n$ extension of $(\sigma_0, \sigma_1, X)$ such that $d^{[i]} \Vdash (\forall x)\neg \Phi_{e_i}(G, x)$. Finding the right $Z_i$ is a $\Pi^0_1(C_{n-1} \oplus (X \oplus Z_i \oplus M_{n-1})')$ event. It can the be found uniformly in any set $P$ which is PA over $\halt^{(n+1)}$. This completes the proof of the lemma.
\end{proof}

We are now ready to derive our main theorems

\subsection{Preservation of non-$\Sigma^0_n$ definitions}

Our first application shows the existence, for every instance of the pigeonhole principle, of a solution which does not collapse the definition of a non-$\Sigma^0_n$ set into a $\Sigma^0_n$ one. This corresponds to preservation of one non-$\Sigma^0_n$ definition, following the terminology of Wang who showed that given $A$ non $\Sigma^0_n$, any non-empty $\Pi^0_1$ class contains an element $X$ such that $A$ is not $\Sigma^0_n(X)$ ~\cite{Wang2014Definability}.

\begin{theorem}\label{thm:rt12-preservation-non-sigma2}
Fix $n \geq 0$ and let $B$ be a non-$\Sigma^0_{n+1}$ set. For every set $A$, there is an infinite set $G \subseteq A$ or $G \subseteq \overline{A}$ such that $B$ is not $\Sigma^{0}_{n+1}(G)$.
\end{theorem}
\begin{proof}
We let $A^0 = \overline{A}$ and $A^1 = A$. We work with the $\Qb_n$ forcing. By Wang~\cite[Theorem 3.6.]{Wang2014Definability}, we can also assume that $B$ is not $\Sigma^0_1(\Mcal_n)$. We also suppose $n>0$, the case $n=0$ was proved by Dzhafarov and Jockusch \cite{Dzhafarov2009Ramseys}.

\textbf{The asymmetric case:}
Suppose first there exists a $\Qb_n$-condition $b = (\sigma_0, \sigma_1, X)$ which is invalid for some side $i-1 < 2$. Let $\Ucal \subseteq \langle \Ucal^{\Mcal_{n-1}}_{C_{n-1}} \rangle$ be a $\Sigma^0_1(\halt^{(n)})$ largeness class such that $X \cap A^i \notin \Ucal$. Note that every condition $c \leq b$ must be valid for side $i$ as otherwise we would have $Y \cap A^0$ and $Y \cap A^1$ both not in $\langle \Ucal^{\Mcal_{n-1}}_{C_{n-1}} \rangle$ for some $Y \in \langle \Ucal^{\Mcal_{n-1}}_{C_{n-1}} \rangle$ which would contradict that $\langle \Ucal^{\Mcal_{n-1}}_{C_{n-1}} \rangle$ is partition regular.

We then work below $b = (\sigma_0, \sigma_1, X)$. Given $c \leq b$ and a $\Sigma^0_{n+1}$ statement $(\exists x) \Phi_e(G, x, y)$ with one free variable $y$, the set $\{y\ :\ c \qvdash^\Ucal (\exists x) \Phi_e(G, x, y)\}$ is $\Sigma^0_1(M_n)$ from \Cref{noideaeffploufplouf}. As $B$ is not $\Sigma^0_1(M_n)$ there exists $y \in B$ such that $d \nqvdash^\Ucal (\exists x) \Phi_e(G, x, y)$ or there exists $y \notin B$ such that $c \qvdash^\Ucal (\exists x) \Phi_e(G, x, y)$. In the first case using \Cref{noideaeffploufplouf2} we find an extension $d \leq c$ such that $d^{[i]} \Vdash (\forall x) \neg \Phi_e(G, x, y)$ and in the second case an extension $d \leq c$ such that $d^{[i]} \Vdash (\exists x) \Phi_e(G, x, y)$.

Now given $c \leq b$ and a $\Sigma^0_{m+1}$ statement $(\exists x) \Phi_e(G, x)$ for $m < n$ we ask if $c^{[i]} \qvdash (\exists x) \Phi_e(G, x)$. Using \Cref{prop-hyp-forcexta-effect} if the answer is positive we find an extension $d \leq c$ such that $d^{[i]} \Vdash (\exists x) \Phi_e(G, x)$ and otherwise we find an extension $d \leq c$ such that $d^{[i]} \Vdash (\forall x) \neg \Phi_e(G, x)$. 

In the end we build a $n-1$-generic filter $\Fcal \subseteq \Pb_{n-1}$ such that $G_\Fcal \subseteq A^i$ with in addition that some $p$ forces $B \neq \{y\ :\ (\exists x) \Phi_e(G, x, y)\}$ for every $\Sigma^0_{n+1}$ statement $(\exists x) \Phi_e(G, x, y)$. By \Cref{lem:forcing-implytruth} we then have that $B \neq \{y\ :\ (\exists x) \Phi_e(G_\Fcal, x, y)\}$ for every $\Sigma^0_{n+1}$ statement $(\exists x) \Phi_e(G, x, y)$. Thus $B$ is not $\Sigma^{0}_{n+1}(G_\Fcal)$.\\

\textbf{The symmetric case:}
Suppose now that every  $\Qb_n$-condition $c = (\sigma_0, \sigma_1, X)$ is valid for both sides. Given a condition $c$ and two $\Sigma^0_{n+1}$ statement $(\exists x) \Phi_{e_0}(G, x, y), (\exists x) \Phi_{e_1}(G, x, y)$ with one free variable $y$, the set $\{y\ :\ c \qvdash (\exists x) \Phi_{e_0}(G, x, y) \vee (\exists x) \Phi_{e_1}(G, x, y)\}$ is $\Sigma^0_1(M_n)$ from \Cref{lem:pb2-forcing-question-complexity}. As $B$ is not $\Sigma^0_1(M_n)$ there exists $y \in B$ such that $c \nqvdash (\exists x) \Phi_{e_0}(G, x, y) \vee (\exists x) \Phi_{e_1}(G, x, y)$ or there exists $y \notin B$ such that $c \qvdash (\exists x) \Phi_{e_0}(G, x, y) \vee (\exists x) \Phi_{e_1}(G, x, y)$. In the first case using \Cref{lem:pb2-forcing-question-spec} we find an extension $d \leq c$ such that $d^{[i]} \Vdash (\forall x) \neg \Phi_{e_i}(G, x, y)$ for some $i < 2$ and in the second case an extension $d \leq c$ such that $d^{[i]} \Vdash (\exists x) \Phi_{e_i}(G, x, y)$ for some $i < 2$.

Now given $c$ and a $\Sigma^0_{m+1}$ statement $(\exists x) \Phi_e(G, x)$ for $m < n$ we find using \Cref{prop-hyp-forcexta-effect} an extension $d \leq c$ such that $d^{[i]} \Vdash (\forall x) \Phi_e(G, x)$ or $d^{[i]} \Vdash (\forall x) \neg \Phi_e(G, x)$ for both $i=0$ and $i=1$.

In the end we have one filter $\Fcal \subseteq \Qb_n$ giving two filters $\Fcal_0, \Fcal_1 \subseteq \Pb_{n-1}$ corresponding to side $0$ and $1$, which are both $n-1$-generic and such that $G_{\Fcal_0} \subseteq A^0$ and  $G_{\Fcal_1} \subseteq A^1$. Also for every $\Sigma^0_{n+1}$ formulas $(\exists x) \Phi_{e_0}(G, x, y), (\exists x) \Phi_{e_1}(G, x, y)$ we have $d \in \Fcal$ such that $d^{[0]}$ forces $B \neq \{y\ :\ (\exists x) \Phi_{e_0}(G, x, y)\}$ or $d^{[1]}$ forces $B \neq \{y\ :\ (\exists x) \Phi_{e_1}(G, x, y)\}$. By a usual pairing argument, there must be $i<2$ such that for every $\Sigma^0_{n+1}$ formula $(\exists x) \Phi_{e}(G, x, y)$ we have $d \in \Fcal$ such that $d^{[i]}$ forces $B \neq \{y\ :\ (\exists x) \Phi_{e}(G, x, y)\}$. By \Cref{lem:forcing-implytruth} we then have that $B \neq \{y\ :\ (\exists x) \Phi_{e}(G_{\Fcal_i}, x, y)\}$ for every such formula and then that $B$ is not $\Sigma^{0}_{n+1}(G_\Fcal)$.
\end{proof}

The following corollary would correspond to strong iterated jump cone avoidance of $\rt^1_2$,
following the terminology of Wang~\cite{Wang2014Some}.

\begin{theoremnonumber}[Main Theorem 1 (Theorem \ref{maintheorem1})]
Fix $n \geq 0$. Let $B$ be non $\emptyset^{(n)}$-computable. Every set $A$ has an infinite subset $H \subseteq A$ or $H \subseteq \overline{A}$ such that $B$ is not $H^{(n)}$-computable.
\end{theoremnonumber}
\begin{proof}
Given a set $B$ which is not $\emptyset^{(n)}$-computable, either $B$ or $\overline{B}$ is not $\Sigma^0_{n+1}$. By Theorem~\ref{thm:rt12-preservation-non-sigma2}, for every set $A$, there is an infinite set $H \subseteq A$ or $H \subseteq \overline{A}$ such that either $B$ or $\overline{B}$ is not $\Sigma^{0}_{n+1}(H)$, hence such that $B$ is not $H^{(n)}$-computable.
\end{proof}

\subsection{Preservation of $\Delta^0_n$ hyperimmunities}

Our second application concerns the ability to prevent solutions from computing fast-growing functions. Recall the definition of hyperimmunity.

\begin{definition}
A function $f$ \emph{dominates} a function $g$ if $f(x) \geq g(x)$
for every $x$. A function $f$ is \emph{$X$-hyperimmune} if it is not dominated by any $X$-computable function.
\end{definition}

The following lemma is proven by Downey et al.~\cite[Lemma 3.3]{Downey2019Relationships}.

\begin{lemma}[\cite{Downey2019Relationships}]\label{lem:k-non-ce-to-k-hyperimmune-preservation-1}
For every $k \le \omega$ and every $Z$, for any nondecreasing functions $(f_i)_{i < k}$ which are $Z$-hyperimmune, there is a $G$ and sets $(A_i)_{i < k}$ such that none of the $A_i$ is $\Sigma^0_1(Z\oplus G)$, but for any $i$ and any function $h$ dominating $f_i$, $A_i$ is $\Sigma^0_1(Z\oplus G \oplus h)$.
\end{lemma}

\begin{theorem}\label{thm:rt12-preservation-delta2-hyperimmunity}
Fix a $\emptyset^{(n)}$-hyperimmune function $f$. For every set $A$, there is an infinite set $H \subseteq A$
or $H \subseteq \overline{A}$ such that $f$ is $H^{(n)}$-hyperimmune.
\end{theorem}
\begin{proof}
By Lemma~\ref{lem:k-non-ce-to-k-hyperimmune-preservation-1}, letting $Z = \emptyset^{(n)}$, there is a set $G$ and a set $B$ such that $B$ is not $\Sigma^0_1(\emptyset^{(n)} \oplus G)$ but for any function $h$ dominating $f$, $B$ is $\Sigma^0_1(\emptyset^{(n)} \oplus G \oplus h)$.
By the jump inversion theorem, there is a set $Q$ such that $Q^{(n)} \equiv_T \emptyset^{(n)} \oplus G$.
In particular, $B$ is not $\Sigma^0_1(Q^{(n)})$, so it is not $\Sigma^0_{n+1}(Q)$.
By Theorem~\ref{thm:rt12-preservation-non-sigma2}, there is an infinite set $H \subseteq H$ or $H \subseteq \overline{A}$ such that $B$ is not $\Sigma^0_{n+1}(H \oplus Q)$. In particular $B$ is not $\Sigma^0_1((H \oplus Q)^{(n)})$ and therefore not $\Sigma^0_1(H^{(n)} \oplus G)$. Suppose for the contradiction that $f$ is dominated by an $H^{(n)}$-computable function $h$. Then $B$ is $\Sigma^0_1(\emptyset^{(n)} \oplus G \oplus h)$, hence $B$ is $\Sigma^0_1(H^{(n)} \oplus G)$. Contradiction.
\end{proof}

\subsection{Low${}_n$ solutions}

An effectivization of the forcing construction enables us to obtain lowness results for the infinite pigeonhole principle. The existence of low${}_2$ solutions for $\Delta^0_2$ sets, and of low${}_2$ cohesive sets for computable sequences of sets, was proven by Cholak, Jockusch and Slaman~\cite[sections 4.1 and 4.2]{Cholak2001strength}. The existence of low${}_3$ cohesive sets for $\Delta^0_2$ sequences of sets was proven by Wang~\cite[Theorem 3.4]{Wang2014Cohesive}. Wang~\cite[Questions 6.1 and 6.2]{Wang2014Cohesive} and the second author~\cite[Question 5.4]{Patey2016Open} asked whether such results can be generalized for every $\Delta^0_{n+1}$ instances of the pigeonhole and every $\Delta^0_n$ instances of cohesiveness. We answer positively both questions.

\begin{theorem}\label{thm:rt12-delta3-PA-double-jump-solution}
Let $n \geq 0$. For every $\halt^{(n+1)}$-computable set $A$ and every $P$ PA over $\emptyset^{(n+1)}$, there is an infinite set $G \subseteq A$ or $G \subseteq \overline{A}$ such that $G^{(n+1)} \leq_T P$.
\end{theorem}
\begin{proof}
The case $n = 0$ is proven by Cholak, Jockusch and Slaman~\cite[sections 4.1 and 4.2]{Cholak2001strength}. Suppose $n > 0$. Fix $P$ and $A$, and let $A^0 = \overline{A}$ and $A^1 = A$. We work with the $\Qb_n$ forcing. We again have two constructions, based on whether every condition have both valid sides or not.\\

\textbf{Asymmetric case:} 
Suppose first there exists a $\Qb_n$-condition $b = (\sigma_0, \sigma_1, X)$ which is invalid for some side $i-1 < 2$. Let $\Ucal \subseteq \langle \Ucal^{\Mcal_{n-1}}_{C_{n-1}} \rangle$ be a $\Sigma^0_1(\halt^{(n)})$ largeness class such that $X \cap A^i \notin \Ucal$. Note that every condition $c \leq b$ must be valid for side $i$ as otherwise we would have $Y \cap A^0$ and $Y \cap A^1$ both not in $\langle \Ucal^{\Mcal_{n-1}}_{C_{n-1}} \rangle$ for some $Y \in \langle \Ucal^{\Mcal_{n-1}}_{C_{n-1}} \rangle$ which would contradict that $\langle \Ucal^{\Mcal_{n-1}}_{C_{n-1}} \rangle$ is partition regular. We then work below $b = (\sigma_0, \sigma_1, X)$. 

Given $c \leq b$ and a $\Sigma^0_{n+1}$ statement $(\exists x) \Phi_e(G, x, y)$ with one free variable $y$, we ask if $c \qvdash^\Ucal (\exists x) \Phi_e(G, x, y)$. From \Cref{noideaeffploufplouf} we obtain the answer uniformly in $\halt^{(n+1)}$ and thus uniformly in $P$. If the answer is yes, from \Cref{noideaeffploufplouf2} we find uniformly in $P$ an extension $d \leq c$ such that $d^{[i]} \Vdash (\forall x) \neg \Phi_e(G, x, y)$ and in the second case an extension $d \leq c$ such that $d^{[i]} \Vdash (\exists x) \Phi_e(G, x, y)$.

Now given $c \leq b$ and a $\Sigma^0_{m+1}$ statement $(\exists x) \Phi_e(G, x)$ for $m < n$ we ask if $c^{[i]} \qvdash (\exists x) \Phi_e(G, x)$. From \Cref{prop-hyp-effectivalla} we obtain the answer uniformly in $\halt^{(n+1)}$ and thus uniformly in $P$. Using \Cref{prop-hyp-forcexta-effect} if the answer is positive we find uniformly in $P$ an extension $d \leq c$ such that $d^{[i]} \Vdash (\exists x) \Phi_e(G, x)$ and otherwise we find an extension $d \leq c$ such that $d^{[i]} \Vdash (\forall x) \neg \Phi_e(G, x)$. 

In the end we build effectively in $P$ a $n$-generic filter $\Fcal \subseteq \Pb_{n-1}$ such that $G_\Fcal \subseteq A^i$. Using \cref{lem:forcing-implytruth} and by construction, $P$ can also decide every $\Sigma^0_{n+1}(G_\Fcal)$ statement. Thus $(G_\Fcal)^{(n+1)} \leq_T P$.

\textbf{The symmetric case:}
Suppose now that every  $\Qb_n$-condition $c = (\sigma_0, \sigma_1, X)$ is valid for both sides. Given a condition $c$ and two $\Sigma^0_{n+1}$ statement $(\exists x) \Phi_{e_0}(G, x, y), (\exists x) \Phi_{e_1}(G, x, y)$ with one free variable $y$, we ask if $c \qvdash (\exists x) \Phi_{e_0}(G, x, y) \vee (\exists x) \Phi_{e_1}(G, x, y)$. From \Cref{lem:pb2-forcing-question-complexity} we obtain the answer uniformly in $\halt^{(n+1)}$ and then uniformly in $P$. Using \Cref{lem:pb2-forcing-question-spec} we find uniformly in $P$ an extension $d \leq c$ such that $d^{[i]} \Vdash (\exists x) \Phi_{e_i}(G, x, y)$ or $d^{[i]} \Vdash (\forall x) \neg \Phi_{e_i}(G, x, y)$ for some $i<2$.

Now given $c$ and a $\Sigma^0_{m+1}$ statement $(\exists x) \Phi_e(G, x)$ for $m < n$ we ask if $c^{[0]} \qvdash (\exists x) \Phi_e(G, x)$. From \Cref{prop-hyp-effectivalla} we obtain the answer uniformly in $\halt^{(n+1)}$ and thus uniformly in $P$. Using \Cref{prop-hyp-forcexta-effect} if the answer is positive we find uniformly in $P$ an extension $d \leq c$ such that $d^{[0]} \Vdash (\exists x) \Phi_e(G, x)$ and otherwise we find an extension $d \leq c$ such that $d^{[0]} \Vdash (\forall x) \neg \Phi_e(G, x)$. We then ask whether $d^{[1]} \qvdash (\exists x) \Phi_e(G, x)$ and find similarly an extension $h \leq d$ such that $h^{[1]} \Vdash (\forall x) \neg \Phi_e(G, x)$ or $h^{[1]} \qvdash (\exists x) \Phi_e(G, x)$.

In the end we build effectively in $P$ a filter $\Fcal \subseteq \Qb_n$ giving two filters $\Fcal_0, \Fcal_1 \subseteq \Pb_{n-1}$ corresponding to side $0$ and $1$, which are both $n-1$-generic and such that $G_{\Fcal_0} \subseteq A^0$ and  $G_{\Fcal_1} \subseteq A^1$. By a pairing argument there must be $i<2$ such that $\Fcal_i$ is $n$-generic. Using \cref{lem:forcing-implytruth} and by construction, $P$ can decide every $\Sigma^0_{n+1}(G_{\Fcal_i})$ statement. Thus $(G_{\Fcal_i})^{(n+1)} \leq_T P$.
\end{proof}

\begin{theoremnonumber}[Main theorem 3 (Theorem \ref{maintheorem3})]
Fix $n \geq 0$. Every $\halt^{(n+1)}$-computable set $A$ has an infinite subset $H \subseteq A$ or $H \subseteq \overline{A}$ of low${}_{n+2}$ degree.
\end{theoremnonumber}
\begin{proof}
 By the relativized low basis theorem~\cite{Jockusch197201}, there is some $P$ PA over $\emptyset^{(n+1)}$ such that $P' \leq_T \emptyset^{(n+2)}$. By Theorem~\ref{thm:rt12-delta3-PA-double-jump-solution}, there is an infinite set $G \subseteq A$ or $G \subseteq \overline{A}$ such that $G^{(n+1)} \leq_T P$. In particular, $G^{(n+2)} \leq_T P' \leq_T \emptyset^{(n+2)}$. Thus $G$ is of low${}_{n+2}$ degree.
\end{proof}

\section{Arithmetic and Hyperarithmetic cone avoidance}\label{sect:hypreduction}
In this section, we extend the jump control of solutions to the pigeonhole principle to ordinal iterations of the jump. We then derive a proof of strong cone avoidance for arithmetic and hyperarithmetic reductions. We prove in the mean time cone avoidance for arithmetical reductions. The reader already familiar with higher recursion theory may jump directly to \cref{sec_strategy_coneavoid} where we give the general strategy which will be used to show hyperarithmetic cone avoidance.

\subsection{Background on higher recursion theory}

\subsubsection{Computable ordinals}

We let $\wck$ denote the first non-computable ordinal. There is a $\Pi^1_1$ set $\Ocal_1 \subseteq \omega$ such that each $o \in \Ocal_1$ codes for an ordinal $\alpha < \wck$ and each ordinal $\alpha < \wck$ has a unique code in $\Ocal_1$. Furthermore given that $o \in \Ocal_1$, one can computably recognize if $o$ codes for $0$, if $o$ codes for a successor ordinal $\alpha+1$, in which case we can uniformly and computably produce a code in $\Ocal_1$ for $\alpha$, and if $o$ codes for a limit ordinal $\sup_n \beta_n$, in which case we can uniformly and computably produce for each $n$ codes in $\Ocal_1$ for $\beta_n$. See \cite{Sacks1990Higher} for more details about $\Ocal_1$. In this section, we manipulate each ordinal $\alpha < \wck$ via its respective code in $\Ocal_1$. To simplify the reading, we use the notation $\alpha$ instead of the code for $\alpha$.

\subsubsection{The effective Borel sets}

We also use codes for effective Borel subsets of $\omega$ or of $2^\omega$ : For $\alpha < \wck$ a code for a $\Sigma^0_{\alpha+1}$ set $\Bcal = \bigcup_{n < \omega} \Bcal_{n}$ is the code of a function that effectively enumerate codes for each $\Pi^0_{\alpha}$ set $\Bcal_{n}$. A code for a $\Pi^0_{\alpha+1}$ set $\Bcal = \bigcap_{n < \omega} \Bcal_{n}$ is the code of a function that effectively enumerate codes for each $\Sigma^0_{\alpha}$ set $\Bcal_{n}$. For $\alpha = \sup_n \beta_n$ limit a code of a $\Sigma^0_{\alpha}$ set $\Bcal = \bigcup_{n < \omega} \Bcal_{\beta_n}$ is the code of a function that effectively enumerate codes for each $\Pi^0_{\beta_n}$ set $\Bcal_{\beta_n}$ with $\sup_n \beta_n = \alpha$. The code of a $\Pi^0_{\alpha}$ set $\Bcal = \bigcap_{n < \omega} \Bcal_{\beta_n}$ is the code of a function that effectively enumerate codes for each $\Sigma^0_{\beta_n}$ set $\Bcal_{\beta_n}$ with $\sup_n \beta_n = \alpha$. We also assume the codes for effective Borel sets include some information so that we can computably distinguish $\Pi^0_\alpha$ from $\Sigma^0_\alpha$ codes as well as distinguish if $\alpha=1$, if $\alpha$ is successor or if it is limit.

\subsubsection{The iterated jumps}

We use such codes to iterate the jump through the ordinals:
\begin{enumerate}
\item $\halt^{(0)} = \emptyset$
\item $\halt^{(\alpha+1)} = (\halt^{(\alpha)})'$
\item $\halt^{(\sup_n \alpha_n)} = \oplus_{n \in \omega} \halt^{(\alpha_n)}$
\end{enumerate}

Note that for $n < \omega$ the set $\halt^{(n)}$ is $\Sigma^0_n$ and complete for $\Sigma^0_n$ questions. Above the first limit ordinal the situation is slightly different : $\halt^{(\omega)}$ is $\Delta^0_{\omega}$ and not $\Sigma^0_{\omega}$. Also given $\alpha \geq \omega$ we have that $\halt^{(\alpha+1)}$ is $\Sigma^0_{\alpha}$ and complete for $\Sigma^0_{\alpha}$ questions.

\begin{proposition} \label{lemma-hyp-eff4}
Let $n \in \omega$.
\begin{enumerate}
\item Let $m > 0$. The set $\{X\ :\ n \in X^{(m)}\}$ is a $\Sigma^0_{m}$ class.
\item Let $\alpha$ be limit. The set $\{X\ :\ n \in X^{(\alpha)}\}$ is a $\Delta^0_{\beta}$ class for some $\beta < \alpha$.
\item Let $\alpha = \beta + 1$ with $\beta \geq \omega$. The set $\{X\ :\ n \in X^{(\alpha)}\}$ is a $\Sigma^0_{\beta}$ class.
\end{enumerate}
\end{proposition}
\begin{proof}
The set $\{X\ :\ n \in X'\}$ is clearly $\Sigma^0_1$. Let $m > 1$. the set $\{X\ :\ n \in X^{(m)}\}$ equals
$$\bigcup_{\{\sigma\ :\ \Phi_n(\sigma, n)\downarrow\}} \bigcap_{\{i\ :\ \sigma(i) = 0\}} \{X\ :\ i \notin X^{(m-1)}\} \cap \bigcap_{\{i\ :\ \sigma(i) = 1\}} \{X\ :\ i \in X^{(m-1)}\}$$
This is by induction a $\Sigma^0_{m}$ set.

Let $\alpha$ be limit. Let $p_1,p_2$ be projections of the pairing function, that is, $x = \langle p_1(x), p_2(x)\rangle$. Then $\{X\ :\ n \in X^{(\alpha)}\}$ equals $\{X\ :\ p_1(n) \in X^{(p_2(n))}\}$, which is a $\Delta^0_\beta$ set for $\beta < \alpha$.

Let $\alpha = \beta+1$. The set $\{X\ :\ n \in X^{(\beta+1)}\}$ equals
$$\bigcup_{\{\sigma\ :\ \Phi_n(\sigma, n)\downarrow\}} \bigcap_{\{i\ :\ \sigma(i) = 0\}} \{X\ :\ i \notin X^{(\beta)}\} \cap \bigcap_{\{i\ :\ \sigma(i) = 1\}} \{X\ :\ i \in X^{(\beta)}\}$$
This is by induction a $\Sigma^0_{\beta}$ class.
\end{proof}

\begin{proposition} \label{lemma-hyp-eff3}
Let $\Phi$ be a functional. Let $n,i \in \omega$.
\begin{enumerate}
\item Let $m > 0$. The set $\{X\ :\ \exists t\ \Phi(X^{(m)}, n)[t]\downarrow = i\}$ is a $\Sigma^0_{m+1}$ class.
\item Let $\alpha \geq \omega$. The set $\{X\ :\ \exists t\ \Phi(X^{(\alpha)}, n)[t]\downarrow = i\}$ is a $\Sigma^0_{\alpha}$ class.
\end{enumerate}
\end{proposition}
\begin{proof}
Trivial using \Cref{lemma-hyp-eff4}
\end{proof}

\subsubsection{$\Pi^1_1$ and $\Sigma^1_1$ sets of integers}

We previously mentioned a $\Pi^1_1$ set $\Ocal_1$ of unique notations for ordinals. This set is included in Kleene's $\Ocal$, the set of all the constructible codes for the computable ordinals. Given an ordinal $\alpha < \wck$, let $\Ocal_{< \alpha}$ denote the elements of $\Ocal$ which code for an ordinal strictly smaller than $\alpha$. Each $\Ocal_{<\alpha}$ is $\Delta^1_1$ uniformly in $\alpha$ (it actually is always a $\Sigma^0_{\alpha+1}$ set \cite{monin2014higherthesis}). It is well-known that $\Ocal$ is a $\Pi^1_1$-complete set \cite{Sacks1990Higher}, that is, for any $\Pi^1_1$ set $B \subseteq \omega$ there is a computable function $f:\omega \rightarrow \omega$ such that $n \in B \leftrightarrow f(n) \in \Ocal$. Let us define $B_\alpha = \{n\ :\ f(n) \in \Ocal_{<\alpha}\}$. In particular, each $B_{\alpha}$ is $\Delta^1_1$ uniformly in $\alpha$ and $B = \bigcup_{\alpha < \wck} B_\alpha$. In particular $B$ is a $\Sigma^0_{\wck}$ set. Note that contrary to $\Sigma^0_\alpha$ sets for $\alpha < \wck$, the $\Sigma^0_{\wck}$ are not described with a computable code, but rather with a $\Pi^1_1$ set of codes for all the $\Pi^0_{\alpha}$ that constitutes the $\Sigma^0_{\wck}$ set $B$. With a little hack, we can even make sure that at most one new element appears in each $B_\alpha$. For this reason, we often see $\Pi^1_1$ sets as enumerable along the computable ordinals.

By complementation a $\Sigma^1_1$ set $B \subseteq \omega$ can be seen as co-enumerable along the computable ordinals and we have $B = \bigcap_{\alpha < \wck} B_\alpha$ where each $B_\alpha$ is $\Delta^1_1$ uniformly in $\alpha$. We also say in this case that $B$ is $\Pi^0_{\wck}$.

\subsubsection{$\Sigma^1_1$-boundedness}

A central theorem when working with $\Sigma^1_1$ and $\Pi^1_1$ sets is $\Sigma^1_1$-boundedness:

\begin{theorem}[$\Sigma^1_1$-boundedness \cite{spector1955recursive}]
Let $B$ be a $\Sigma^1_1$ set of codes for ordinals, then the supremum of the ordinals coded by elements of $B$ is strictly smaller than $\wck$.
\end{theorem}

We mostly here use the following corollary:

\begin{corollary}
Let $f:\omega \to \wck$ be a total $\Pi^1_1$ function. Then $\sup_n f(n) = \alpha < \wck$.
\end{corollary}

Note that $f:\omega \to \wck$ means the range of $f$ is a subset of $\Ocal_1$. The corollary comes from the fact that if $f$ is total, then it becomes $\Delta^1_1$ and its range is then a $\Sigma^1_1$ set of codes for ordinals. As an example we apply here $\Sigma^1_1$-boundedness to show a simple fact that will be needed later : adding an $\omega$-bounded quantifier to a $\Sigma^0_{\wck}$ or a $\Pi^0_{\wck}$ set does not change its complexity.

\begin{lemma}
Every $\Sigma^0_{\wck + 1}$ set of integers is $\Pi^0_{\wck}$.
\end{lemma}
\begin{proof}
Let $B$ be $\Sigma^0_{\wck + 1}$, that is, $B = \bigcup_{n \in \omega} \bigcap_{\alpha \in \wck} B_{n, \alpha}$ where each $B_{n, \alpha}$ is $\Sigma^0_{\alpha}$ uniformly in $\alpha$. Then $B$ is $\Pi^0_{\wck}$ via the following equality : $\bigcup_{n \in \omega} \bigcap_{\alpha \in \wck} B_{n, \alpha} = \bigcap_{\alpha \in \wck} \bigcup_{n \in \omega} \bigcap_{\beta \in \alpha} B_{n, \beta}$.
\end{proof}

It is clear that if $m$ is in the leftmost set it is also in the rightmost set. The reader should have no trouble to apply $\Sigma^1_1$-boundedness to show that if $m$ is not in the leftmost set, then it is not in the rightmost one.

\subsubsection{$\Pi^1_1$ and $\Sigma^1_1$ sets of reals}

Given $X \in 2^{\omega}$ we let $\Ocal^X$ be the set of $X$-constructible codes for $X$-computable ordinals. We let $\omega_1^X \geq \wck$ be the smallest non $X$-computable ordinal. For $\alpha < \omega_1^X$, we let $\Ocal^X_{<\alpha}$ be the elements of $\Ocal^X$ coding for an ordinal strictly smaller than $\alpha$.

One can show that a set $\Bcal \subseteq 2^{\omega}$ is $\Pi^1_1$ iff there exists some $e \in \omega$ such that $\Bcal = \{X\ :\ e \in \Ocal^X\}$, that is, $\Bcal$ is the set of elements relative to which $e$ codes for an $X$-computable ordinal. In particular, $\Bcal = \bigcup_{\alpha < \omega_1} \{X\ :\ e \in \Ocal_{<\alpha}^X\}$. Note that the union may go up to $\omega_1$, indeed, $\Pi^1_1$ sets of reals are not necessarily Borel.

A $\Pi^1_1$ set of particular interest is the set of element $X$ such that $\omega_1^X > \wck$. The set is Borel, but not effectively. One can even prove that it contains no non-empty $\Sigma^1_1$ subset : this is known as the Gandy Basis theorem (see Sacks \cite[III.1.5]{Sacks1990Higher}):

\begin{theorem}[Gandy Basis theorem]
Let $\Bcal \subseteq 2^{\omega}$ be a non-empty $\Sigma^1_1$ set. Then there exists $X \in \Bcal$ such that $\omega_1^X = \wck$.
\end{theorem}

\subsubsection{The general strategy to show hyperarithmetic cone avoidance} \label{sec_strategy_coneavoid}

Let $Z$ be non $\Delta^1_1$. Our goal is to build a generic $G \subseteq A$ or $G \subseteq \omega - A$ such that $Z$ is not $\Delta^1_1(G)$. This is done in two steps: first show that $Z$ is not $G^{(\alpha)}$-computable for any $\alpha < \wck$ and second show that $\omega_1^{G} = \wck$, so in particular we cannot have that $Z$ is $G^{(\alpha)}$-computable for $\wck \leq \alpha < \omega_1^{G}$.

The first part is simply an iteration of the forcing through the computable ordinals, and raises no particular issue. This is done in \Cref{sec-hyp-theforcing}.

The second part is a little bit trickier but still follows a canonical technique, which has often been used, up to some cosmetic changes in its presentation, to show this kind of preservation theorem (see for instance \cite{greenberg2017higher}, \cite{sacks1969measure} or \cite{tanaka1967basis}) : Suppose $\omega_1^G > \wck$, in particular there is an element $e \in \Ocal^G$ which codes for $\wck$, that is $e$ is the code of a functional with $\forall n\ \Phi_e(G, n) \downarrow \in \Ocal_{<\wck}^G$ with $\sup_n |\Phi_e(G,n)| = \wck$ where $|\Phi_e(G,n)|$ is the ordinal coded by $\Phi_e(G,n)$. All we have to do is to show that such a code $e$ does not exist. Given $e$ we show that one of the following holds:
\begin{enumerate}
\item $\exists n\ \forall \alpha < \wck\ \Phi_e(G,n) \notin \Ocal^G_{<\alpha}$
\item $\exists \alpha < \wck\ \forall n\ \Phi_e(G,n) \in \Ocal^G_{<\alpha}$
\end{enumerate}
Each set $\{X\ :\ \Phi_e(X,n) \notin \Ocal^X_{<\alpha}\}$ is $\Delta^1_1$ uniformly in $\alpha$. It follows that the set $\{X\ :\ \exists n\ \forall \alpha < \wck\ \Phi_e(X,n) \notin \Ocal^X_{<\alpha}\}$ is a $\Sigma^0_{\wck + 1}$ set of reals. Contrary to $\Sigma^0_{\wck + 1}$ sets of integers, such sets cannot be simplified. We are then required to extend our forcing questions in order to control the truth of $\Sigma^0_{\wck + 1}$-statements. This is what will be done in \Cref{sec-hyp-preserv}.


\subsection{Preliminaries}

We now design a notion of forcing for controlling the $\alpha$-jump of solutions
to the pigeonhole principle. Unlike the notion of forcing for controlling finite iterations of the jump, this notion is non-disjunctive and initially fixes the side of the instance $A$ from which we will construct a solution. This is at the cost of a forcing question whose definitional complexity is higher than the question it asks.

\begin{proposition} \label{prop-hyp-scott2}
There is a sequence of sets $\{M_\alpha\}_{\alpha < \wck}$ such that:
\begin{enumerate}
\item $M_\alpha$ codes for a countable Scott set $\Mcal_\alpha$
\item $\halt^{(\alpha)}$ is uniformly coded by an element of $\Mcal_\alpha$
\item Each $M_\alpha'$ is uniformly computable in $\halt^{(\alpha+1)}$
\end{enumerate}
\end{proposition}
\begin{proof}
In the proof of \Cref{prop-hyp-scott2} we show how to build a functional $\Phi : 2^\omega \rightarrow 2^\omega$ such that for any oracle $X$, we have that $M' = \Phi(X')$ is such that $M = \oplus_{n \in \omega} X_n$ codes for a Scott set $\Mcal$ with $X_0 = X$.

We simply use here this functionnal with any $\halt^{(\alpha+1)}$ for $\alpha < \wck$.
\end{proof}

Note $\halt^{(\beta)}$ is computable in $\halt^{(\alpha)}$ for $\beta < \alpha$ in a uniform way : there is a unique computable function $f(\halt^{(\alpha)}, \alpha, \beta)$ which outputs $\halt^{(\beta)}$ for every $\beta < \alpha$. Also \Cref{prop-hyp-scott2} implies that $M_\beta$ is computable in $\halt^{(\alpha)}$ for $\beta < \alpha$ and similarly, the computation is uniform in $\beta, \alpha$.

We now turn to an extention of \cref{prop-hyp-cohesiveclassa} to the computable ordinals, for which we reuse \cref{lem-hyp-cohesiveclass1} and \cref{lem-hyp-cohesiveclass2}.

\begin{proposition} \label{prop-hyp-cohesiveclass}
There is a sequence of sets $\{C_\alpha\}_{\alpha < \wck}$ such that:
\begin{enumerate}
\item $\Ucal_{C_\alpha}^{\Mcal_{\alpha}}$ is an $\Mcal_\alpha$-cohesive largeness class
\item $\beta < \alpha$ implies $\Ucal_{C_\alpha}^{\Mcal_{\alpha}} \subseteq \langle \Ucal_{C_\beta}^{\Mcal_{\beta}} \rangle$
\item Each $C_\alpha$ is coded by an element of $\Mcal_{\alpha + 1}$ uniformly in $\alpha$ and $M_{\alpha + 1}$.
\end{enumerate}
\end{proposition}
\begin{proof}
Let $X_i^\alpha$ be the element of $\Mcal_\alpha$ of code $i$, so that each $M_\alpha = \oplus_{i} X_i^\alpha$. Let us argue that there is a computable function $f:\wck \times \wck \times \omega$ such that whenever $\beta < \alpha$, then $X_{i}^\beta = X_{f(\alpha, \beta, i)}^\alpha$: Given an ordinal $\alpha$ the function $f$ considers the $M_\alpha$-code of $\halt^{(\alpha)}$ (which is uniformly coded in $\Mcal_{\alpha}$) and uses it produce an $M_\alpha$-code of $M_\beta = \oplus_{i} X_i^\beta$ (as $M_\beta$ is computable in $\halt^{(\alpha)}$, uniformly in $\beta,\alpha$) and then returns an $M_\alpha$-code of $X_i^\beta$. Given $\alpha < \beta$ and $C \subseteq \omega^2$, we then let $g(\alpha, \beta, C) = \{\langle e, f(\alpha, \beta, i)\rangle : \langle e, i \rangle \in C \}$. In particular, $\Ucal^{\Mcal_{\alpha}}_{g(\alpha, \beta, C)} = \Ucal^{\Mcal_\beta}_C$.

Suppose that stage $\alpha$ we have defined by induction sets $C_\beta$ for each $\beta < \alpha$, verifying $(1) (2)$ and $(3)$. Let us proceed and define $C_\alpha$.

Suppose first that $\alpha = \beta + 1$ is successor. Note that the set $C_{\beta}$ is coded by an element of $\Mcal_{\beta + 1}$ uniformly in $\beta$, and thus that $C_{\beta}$ is uniformly computable in $\halt^{(\beta+2)}$ and then uniformly computable in $M_\beta''$. Using \Cref{lem-hyp-cohesiveclass1} we define $D_\beta \supseteq C_\beta$ to be such that $\Ucal_{D_\beta}^{\Mcal_\beta} = \langle \Ucal_{C_\beta}^{\Mcal_\beta} \rangle$ and such that $D_\beta$ is uniformly $M_{\beta}''$-computable. We define $E_\alpha$ to be $g(\alpha, \beta, D_{\beta})$, so that $\Ucal_{E_\alpha}^{\Mcal_\alpha} = \Ucal_{D_\beta}^{\Mcal_\beta}$. Note that as $E_\alpha$ is uniformly computable in $M_{\beta}''$ and thus in $\halt^{(\alpha+1)}$, it is uniformly coded by an element of $\Mcal_{\alpha+1}$. Note also that $\Ucal_{E_\alpha}^{\Mcal_\alpha}$ is partition regular as it equals $\langle \Ucal_{C_\beta}^{\Mcal_\beta} \rangle$. Using \Cref{lem-hyp-cohesiveclass2} we uniformly find an $\Mcal_{\alpha+1}$-index of $C_\alpha \supseteq E_\alpha$ to be such that $\Ucal_{C_\alpha}^{\Mcal_{\alpha}}$ is an $\Mcal_\alpha$-cohesive largeness class.

At limit stage $\alpha = \sup_n \beta_n$, each set $C_{\beta_n}$ is coded by an element of $\Mcal_{\beta_n + 1}$ uniformly in $\beta_n$ and that $\Mcal_{\beta_n + 1}$ is uniformly computable in $\halt^{(\alpha)}$. It follows that $\bigcup_{n} C_{\beta_n}$ is uniformly computable in $\halt^{(\alpha)}$. We define $D_\alpha$ to be $\bigcup_n g(\alpha, \beta_n, C_{\beta_n})$. Note that $D_\alpha$ is uniformly computable in $\halt^{(\alpha)}$ and thus coded by an element of $\Mcal_\alpha$ uniformly in $\alpha$. Note also that $\Ucal_{D_\alpha}^{M_\alpha} = \bigcap_{n \in \omega} \Ucal_{C_{\beta_n}}^{M_{\beta_n}} = \bigcap_{n \in \omega} \langle \Ucal_{C_{\beta_n}}^{M_{\beta_n}} \rangle$. As an intersection of partition regular class, $\Ucal_{D_\alpha}^{M_\alpha}$ is partition regular. Using \Cref{lem-hyp-cohesiveclass2} there is a set $C_\alpha \supseteq D_\alpha$ such that $\Ucal_{C_\alpha}^{\Mcal_\alpha}$ is $\Mcal_\alpha$-cohesive and such that $C_\alpha$ is uniformly coded by an element of $\Mcal_{\alpha+1}$.
\end{proof}

\subsection{The forcing} \label{sec-hyp-theforcing}

From now on, fix sequences $\{\Mcal_\alpha\}_{\alpha < \wck}$ and $\{C_\alpha\}_{\alpha < \wck}$ which verify \Cref{prop-hyp-scott2} and \Cref{prop-hyp-cohesiveclass}, respectively. Assume also that we have a class $\Scal \subseteq \bigcap_{\beta < \wck} \Ucal_{C_\beta}^{\Mcal_{\beta}}$ which is partition regular and that will be detailed later.


Let $A^0 \cup A^1 = \omega$. Note that there must be $i< 2$ such that $A^i \in \Scal$. Let then $A = A^i$ for some $i$ such that $A^i \in \Scal$.

\begin{definition}
Let $\Pb_{\wck}$ be the set of conditions $(\sigma, X)$ such that:
\begin{enumerate}
\item $\sigma \subseteq A$
\item $X \subseteq A$
\item $X \cap \{0, \dots, |\sigma|\} = \emptyset$.
\item $X \in \Scal$
\end{enumerate}
Given two conditions $(\sigma, X), (\tau, Y) \in \Pb_{\wck}$ we let $(\sigma, X) \leq (\tau, Y)$ be the usual Mathias extension, that is, $\sigma \succeq \tau$, $X \subseteq Y$ and $\sigma - \tau \subseteq Y$.
\end{definition}

We now define an abstract forcing question for $\Sigma^0_\alpha$ sets, which is merely an extension of the forcing question of the $\Pb_n$ forcing for $\Sigma^0_{n+1}$ sets : when $\alpha < \omega$, the definition below is merely a reformulation of \Cref{def-hyp-forcingqua} with the use of effective Borel sets instead of formulas.

\begin{definition} \label{def-hyp-forcingqu}
Let $\sigma \in 2^{<\omega}$. Given a $\Sigma^0_1$ class $\Ucal$, let $\sigma \qvdash \Ucal$ hold if
$$\{Y\ :\ \exists \tau \subseteq Y - \{0, \dots, |\sigma|\}\ [\sigma \cup \tau] \subseteq \Ucal\} \cap \Ucal_{C_{0}}^{\Mcal_{0}}$$
is a largeness class. Then inductively, given a $\Sigma^0_m$ class $\Bcal = \bigcup_{n < \omega} \Bcal_{n}$ with $1 < m < \omega$, we let $\sigma \qvdash \Bcal$ hold if
$$\{Y\ :\ \exists \tau \subseteq Y - \{0, \dots, |\sigma|\}\ \exists n\ \sigma \cup \tau \nqvdash 2^\omega - \Bcal_{n}\} \cap \Ucal_{C_{m-1}}^{\Mcal_{m-1}}$$
is a largeness class. Then inductively, given a $\Sigma^0_\alpha$ class $\Bcal = \bigcup_{n < \omega} \Bcal_{\beta_n}$ with $\omega \leq \alpha < \wck$, we define $\sigma \qvdash \Bcal$ if
$$\{Y\ :\ \exists \tau \subseteq Y - \{0, \dots, |\sigma|\}\ \exists n\ \sigma \cup \tau \nqvdash 2^\omega - \Bcal_{\beta_n}\} \cap \Ucal_{C_{\alpha}}^{\Mcal_{\alpha}}$$
is a largeness class.

For a condition $p = (\sigma, X) \in \Pb_{\wck}$ and an effectively Borel set $\Bcal$, we write $p \qvdash \Bcal$ if $\sigma \qvdash \Bcal$.
\end{definition}

We shall now study the effectivity of the relation $\qvdash$. To do so we introduce the following notation.

\begin{definition}
Let $\sigma \in 2^{<\omega}$. Given a $\Sigma^0_1$ class $\Bcal$, we write $\Ucal(\Bcal, \sigma)$ for the open set:
$$\{Y\ :\ \exists \tau \subseteq Y - \{0, \dots, |\sigma|\}\ [\sigma \cup \tau] \subseteq \Bcal\}$$
Given a $\Sigma^0_\alpha$ class $\Bcal = \bigcup_{n < \omega} \Bcal_{\beta_n}$ for $1 < \alpha < \wck$ we write $\Ucal(\Bcal, \sigma)$ for the open set:
$$\{Y\ :\ \exists \tau \subseteq Y - \{0, \dots, |\sigma|\}\ \exists n\ \sigma \cup \tau \nqvdash 2^\omega - \Bcal_{\beta_n}\}$$
\end{definition}

\Cref{prop-hyp-effectivalla} settled the complexity of the relation $\qvdash$ by showing that it is $\Pi^0_{1}(C_{m-1} \oplus \halt^{(m)})$ for a $\Sigma^0_m$ class. We extend here the proposition for $\Sigma^0_\alpha$ classes. Note that in the following one might have the false impression that we loose one jump compare to \cref{prop-hyp-effectivalla}. This is due to the fact that for $\alpha \geq \omega$ the $\Sigma^0_\alpha$-complete set is $\halt^{(\alpha+1)}$ and not $\halt^{(\alpha)}$.

\begin{proposition} \label{prop-hyp-effectivall}
Let $\sigma \in 2^{<\omega}$.
\begin{enumerate}
\item Let $\Bcal$ be a $\Sigma^0_m$ class for $0 < m < \omega$
\begin{enumerate}
\item The set $\Ucal(\Bcal, \sigma)$ is an upward-closed $\Sigma^0_1(C_{m-2} \oplus \halt^{(m-1)})$ open set if $m>1$ and an upward-closed $\Sigma^0_1$ open set if $m = 1$.
\item The relation $\sigma \qvdash \Bcal$ is $\Pi^0_1(C_{m-1} \oplus \halt^{(m)})$.
\end{enumerate}

\item Let $\Bcal$ be a $\Sigma^0_\alpha$ class for $\alpha \geq \omega$.
\begin{enumerate}
\item The set $\Ucal(\Bcal, \sigma)$ is an upward closed $\Sigma^0_1(C_{\alpha-1} \oplus \halt^{(\alpha)})$ open set if $\alpha$ is successor and an upward closed $\Sigma^0_1(\halt^{(\alpha)})$ open set if $\alpha$ is limit.
\item The relation $\sigma \qvdash \Bcal$ is $\Pi^0_1(C_\alpha \oplus \halt^{(\alpha+1)})$.
\end{enumerate}
\end{enumerate}
This is uniform in $\sigma$ and a code for the class $\Bcal$.
\end{proposition}
\begin{proof}
(1) was already proved in \Cref{prop-hyp-effectivalla}. We then only prove (2). This is done by induction on the effective Borel codes. Let $\omega \leq \alpha < \wck$. Suppose (a) and (b) are true for any $\omega \leq \beta < \alpha$. Let $\sigma \in 2^{<\omega}$ and let $\Bcal = \bigcup_{n < \omega} \Bcal_{\beta_n}$ be a $\Sigma^0_\alpha$ class. Let
$$\Ucal(\Bcal, \sigma) = \{Y\ :\ \exists \tau \subseteq Y - \{0, \dots, |\sigma|\}\ \exists n\ \sigma \cup \tau \nqvdash 2^\omega - \Bcal_{\beta_n}\}$$

Let us show (a). Suppose first $\alpha$ is limit. For each $n \in \omega$, the class $2^\omega - \Bcal_{\beta_n}$ is a $\Sigma^0_{\beta_n}$ class uniformly in $\sigma \cup \tau$ and in a code for $\Bcal_{\beta_n}$. By induction hypothesis, or by \cref{prop-hyp-effectivalla} in case $\alpha = \omega$, the relation $\sigma \cup \tau \nqvdash 2^\omega - \Bcal_{\beta_n}$ is, in any case, $\Sigma^0_1(\halt^{(\beta_n+2)})$ and thus $\Sigma^0_1(\halt^{(\alpha)})$. It follows that $\Ucal(\Bcal, \sigma)$ is an upward-closed $\Sigma^0_1(\halt^{(\alpha)})$ open set.

Suppose now $\alpha \geq \omega$ with $\alpha = \beta+1$. For each $n$ we have that $2^\omega - \Bcal_{\beta_n}$ is a $\Sigma^0_{\beta}$ class uniformly in $n$. By induction hypothesis, the relation $\sigma \cup \tau \nqvdash 2^\omega - \Bcal_{\beta_n}$ is $\Sigma^0_1(C_\beta \oplus \halt^{(\beta+1)})$. It follows that $\Ucal(\Bcal, \sigma)$ is an upward closed $\Sigma^0_1(C_{\alpha-1} \oplus \halt^{(\alpha)})$ class.

Let us now show (b). Suppose $\alpha \geq \omega$ successor or limit. Then $\Ucal(\Bcal, \sigma) \cap \Ucal_{C_{\alpha}}^{\Mcal_{\alpha}}$ is a largeness class if for all $F \subseteq C_{\alpha}$, the class $\Ucal(\Bcal, \sigma) \cap \Ucal_{F}^{\Mcal_{\alpha}}$ is a largeness class. It is a $\Pi^0_2(M_\alpha)$ statement uniformly in $F$ and then a $\Pi^0_1(M_\alpha')$ statement uniformly in $F$ and then a $\Pi^0_1(\halt^{(\alpha+1)})$ statement uniformly in $F$. It follows that the statement $\Ucal(\Bcal, \sigma) \cap \Ucal_{C_{\alpha}}^{\Mcal_{\alpha}}$ is a largeness class is $\Pi^0_1(C_{\alpha} \oplus \halt^{(\alpha+1)})$.
\end{proof}

We finally extend the forcing relation of \Cref{def:qb2-forcing-relation} to the transfinite.

\begin{definition}
Let $(\sigma, X) \in \Pb_{\wck}$. Let $\Ucal$ be a $\Sigma^0_1$ class. We define
$$
\begin{array}{rcccl}
(\sigma, X)&\Vdash&\Ucal&\leftrightarrow&[\sigma] \subseteq \Ucal\\
(\sigma, X)&\Vdash&2^\omega - \Ucal&\leftrightarrow&\forall \tau \subseteq X\ [\sigma \cup \tau] \nsubseteq \Ucal\\
\end{array}
$$
Then inductively for $\Sigma^0_\alpha$ classes $\Bcal = \bigcup_{n < \omega} \Bcal_{\beta_n}$, we define:
$$
\begin{array}{rcccl}
(\sigma, X)&\Vdash&\Bcal&\leftrightarrow&\exists n\ (\sigma, X) \Vdash \Bcal_{\beta_n}\\
(\sigma, X)&\Vdash&2^\omega - \Bcal&\leftrightarrow&\forall n\ \forall \tau \subseteq X\ \sigma \cup \tau \qvdash 2^\omega - \Bcal_{\beta_n}\\
\end{array}
$$
\end{definition}

Note that the relation $\Vdash$ does not change compare to the arithmetical case : the definition goes through exactly the same way in the transfinite. It is the same for the relation $\qvdash$. For these reasons the following lemmas and propositions and theorems are all proved exactly the same way as for the arithmetical case, only now our set $\Scal$ is included in $\bigcap_{\beta < \wck} \Ucal_{C_\beta}^{\Mcal_{\beta}}$ and not just in $\bigcap_{m < \omega} \Ucal_{C_m}^{\Mcal_{m}}$.

\begin{lemma} \label{lem-hyp-forcepi}
Let $p \in \Pb_{\wck}$. Let $\Bcal = \bigcap_{n < \omega} \Bcal_{\beta_n}$ be a $\Pi^0_\alpha$ class. Then $p \Vdash \bigcap_{n < \omega} \Bcal_{\beta_n}$ iff for every $n \in \omega$ and every $q \leq p$, $q \qvdash \Bcal_{\beta_n}$.
\end{lemma}
\begin{proof}
Same as \Cref{lem-hyp-forcepia}.
\end{proof}

\begin{proposition}
Let $p \in \Pb_{\wck}$. Let $\Bcal$ be an effectively Borel set. If $p \Vdash \Bcal$ and $q \leq p$ then $q \Vdash \Bcal$.
\end{proposition}
\begin{proof}
Same as \Cref{lem:qb2-forcing-closed-under-extension}.
\end{proof}

\begin{proposition} \label{prop-hyp-forcext}
Let $p \in \Pb_{\wck}$. Let $\Bcal = \bigcup_{n < \omega} \Bcal_{\beta_n}$ be a $\Sigma^0_\alpha$ class for $0 < \alpha < \wck$.
\begin{enumerate}
\item Suppose $p \qvdash \Bcal$. Then there exists $q \leq p$ such that $q \Vdash \Bcal$.
\item Suppose $p \nqvdash \Bcal$. Then there exists $q \leq p$ such that $q \Vdash 2^\omega - \Bcal$.
\end{enumerate}
\end{proposition}
\begin{proof}
Same as \Cref{prop-hyp-forcexta}.
\end{proof}

\begin{definition}
Let $\Fcal \subseteq \Pb_{\wck}$ be a sufficiently generic filter. Then there is a unique set $G_\Fcal \in 2^{\omega}$ such that for every $(\sigma, X) \in \Fcal$ we have $\sigma \prec G_\Fcal$.
\end{definition}

\begin{theorem} \label{th-hyp-forcingimpliestruth}
Let $\Fcal \subseteq \Pb_{\wck}$ be a generic enough filter. Let $p \in \Fcal$. Let $\Bcal_{\alpha} = \bigcup_{n < \omega} \Bcal_{\beta_n}$ be a $\Sigma^0_\alpha$ class for $0 < \alpha < \wck$. Suppose $p \Vdash \Bcal_{\alpha}$. Then $G_\Fcal \in \Bcal_{\alpha}$. Suppose $p \Vdash 2^\omega - \Bcal_{\alpha}$. Then $G_\Fcal \in 2^\omega - \Bcal_{\alpha}$.
\end{theorem}
\begin{proof}
Same as \Cref{lem:forcing-implytruth}.
%
%
\end{proof}

We now have all the necessary parts to show arithmetic strong cone avoidance, and more generally $\alpha$ cone avoidance for a limit ordinal $\alpha$.

\begin{theorem} \label{th-alphaconeavoid}
Let $\alpha \leq \wck$ be a limit ordinal. Suppose $Z$ is not $\Delta^0_1(\halt^{(\beta)})$ for every $\beta < \alpha$. Let $\Fcal$ be a sufficiently generic filter. Then for every $\beta < \alpha$, $Z$ is not $\Delta^0_1(G_{\Fcal}^{(\beta)})$.
\end{theorem}
\begin{proof}
Let $\Phi$ be a functional and $\beta < \alpha$. Let $\Bcal^{n} = \{X\ : \Phi(X^{(\beta)}, n) \downarrow\}$. We want to show that $Z \neq \{n\ : G_\Fcal^{(\beta)} \in \Bcal^{n}\}$. From \Cref{lemma-hyp-eff3}, $\Bcal^{n}$ is a $\Sigma^0_{\beta+1}$ set for each $n \in \omega$  ($\Sigma^0_{\beta}$ if $\beta \geq \omega$ and  $\Sigma^0_{\beta+1}$ if $\beta < \omega$).

Let $p \in \Pb_{\wck}$ be a condition. From \Cref{prop-hyp-effectivall}, the set $\{n\ :\ p \qvdash \Bcal^{n}\}$ is $\Pi^0_{1}(\halt^{(\beta+3)})$. As $Z$ is not $\Pi^0_{1}(\halt^{(\beta+3)})$, then there is some $n \in Z$ such that $p \nqvdash \Bcal^{n}$ or some $n \notin Z$ such that $p \qvdash \Bcal^{n}$. In the first case, there is an extension $q \leq p$ such that $q \Vdash 2^\omega - \Bcal^{n}$ for some $n \in Z$. In the second case, there is an extension $q \leq p$ such that $q \Vdash \Bcal^{n}$ for some $n \notin Z$. By \Cref{th-hyp-forcingimpliestruth}, in the first case $\Phi(G_{\Fcal}^{(\beta)}, n) \uparrow$ holds for some $n \in Z$, and in the second case, $\Phi(G_{\Fcal}^{(\beta)}, n) \downarrow$ holds for some $n \notin Z$.

If $\Fcal$ is sufficiently generic, this is true for any $\beta < \alpha$ and any functional $\Phi$. It follows that for any ordinal $\beta$ the set $Z$ is not $\Sigma^0_1(G_{\Fcal}^{(\beta)})$ and thus not $\Delta^0_1(G_{\Fcal}^{(\beta)})$.
\end{proof}

This shows in particular cone avoidance for arithmetic degrees. 

\begin{theoremnonumber}[Main theorem 2 (Theorem \ref{maintheorem2})]
Let $B$ be non arithmetical. Every set $A$ has an infinite subset $H \subseteq A$ or $H \subseteq \overline{A}$ such that $B$ is not arithmetical in $H$.
\end{theoremnonumber}
\begin{proof}
A direct corollary of the above theorem with $\alpha = \omega$.
\end{proof}

In order to show cone avoidance for hyperarithmetic degrees, one should additionally argue that if $\Fcal$ is sufficiently generic, then $\omega_1^{G_{\Fcal}} = \wck$. The remainder of this section is devoted to the proof of this fact.

\subsection{Preservation of hyperarithmetic reductions} \label{sec-hyp-preserv}

We now prove that the infinite pigeonhole principle admits strong cone avoidance for hyperarithmetic reductions.

\begin{definition}
A largeness class $\Acal$ is $\Gamma$-minimal, where $\Gamma$ is a class of complexity, if for every $\Gamma$-open set $\Ucal$ we have $\Acal \cap \Ucal$ large implies $\Acal \subseteq \Ucal$.
\end{definition}

\begin{proposition} \label{prop-hyp-delta-min}
The class $\bigcap_{\alpha < \wck} \Ucal_{C_\alpha}^{\Mcal_{\alpha}}$ is $\Delta^1_1$-minimal.
\end{proposition}
\begin{proof}
For every $\alpha < \wck$ we have that $\halt^{(\alpha)} \in \Mcal_{\alpha}$ and $\bigcap_{\alpha < \wck} \Ucal_{C_\alpha}^{\Mcal_{\alpha}} \subseteq \langle \Mcal_{\alpha} \rangle$ where $\langle \Mcal_{\alpha} \rangle$ is $\Mcal_\alpha$-minimal. As $\halt^{(\alpha)} \in \Mcal_{\alpha}$ we also have that $\langle \Mcal_{\alpha} \rangle$ is minimal for $\Sigma^0_1(\halt^{(\alpha)})$ open sets. It follows that $\bigcap_{\alpha < \wck} \Ucal_{C_\alpha}^{\Mcal_{\alpha}}$ is $\Delta^1_1$-minimal.
\end{proof}

\begin{proposition}
There is a set $C \in \bigcap_{\alpha < \wck} \Ucal_{C_\alpha}^{\Mcal_{\alpha}}$ such that $C$ is $\Delta^1_1$-cohesive and $\omega_1^C = \wck$
\end{proposition}
\begin{proof}
Let us argue that for any upward closed partition regular class $\bigcap_{n < \omega} \Ucal_n$ where each $\Ucal_n$ is open, not necessarily effectively of uniformly, there is a $\Delta^1_1$-cohesive $C$ in $\bigcap_{n < \omega} \Ucal_n$. This is done by Mathias forcing with conditions $(\sigma, X)$ such that $X \cap \{0, \dots, |\sigma|\} = \emptyset$ and such that $X$ is $\Delta^1_1$ with $X \in \bigcap_{n < \omega} \Ucal_n$. Given a condition $(\sigma, X)$ and $n$ we can force the generic to be in $\Ucal_n$ as follows : As $X \in \Ucal_{n}$ we must have that $\sigma \cup X \in \Ucal_{n}$ because $\Ucal_n$ is upward closed. Thus there must be $\tau \subseteq X \cap \{0, \dots, |\sigma|\}$ such that $[\sigma \cup \tau] \subseteq \Ucal_{n}$. As $\bigcap_{n < \omega} \Ucal_n$ contains only infinite set we must have $X - \{0, \dots, \sigma \cup \tau\} \in \bigcap_{n < \omega} \Ucal_n$. Thus $(\sigma \cup \tau, X - \{0, \dots, \sigma \cup \tau\})$ is a valid extension. Let now $Y$ be $\Delta^1_1$. We can force the generic to be included in $Y$ or $\omega - Y$ up to finitely many elements as follow : We have $X \cap Y \in \bigcap_{n < \omega} \Ucal_n$ or $X \cap (\omega - Y) \in \bigcap_{n < \omega} \Ucal_n$. Then $(\sigma, X \cap Y)$ or $(\sigma, X \cap (\omega - Y))$ is a valid extension.

We have that the set $\bigcap_{\alpha < \wck} \Ucal_{C_\alpha}^{\Mcal_{\alpha}}$ is a $\Sigma^1_1$ class which is also upward closed and partition regular. We also have that the class of $\Delta^1_1$-cohesive sets is a $\Sigma^1_1$ class. By the previous argument their intersection is non-empty. By the $\Sigma^1_1$-basis theorem it must contains $C$ with $\omega_1^C = \wck$.
\end{proof}

\begin{lemma} \label{hyp-lemma-minC}
Suppose $C$ is $\Delta^1_1$-cohesive with $C \in \bigcap_{\alpha < \wck} \Ucal_{C_\alpha}^{\Mcal_{\alpha}}$. Let $\Ucal$ be a $\Delta^1_1$ open set. If $\Lcal_C \cap \Ucal$ is a largeness class, then $\bigcap_{\alpha < \wck} \Ucal_{C_\alpha}^{\Mcal_{\alpha}} \subseteq \Ucal$
\end{lemma}
\begin{proof}
Suppose $\Lcal_C \cap \Ucal$ is a largeness class. Let us show that $\Ucal \cap \bigcap_{\alpha < \wck} \Ucal_{C_\alpha}^{\Mcal_{\alpha}}$ is a largeness class. Suppose first for contradiction that it is not. Then there is a $\Delta^1_1$ cover $Y_0 \cup \dots \cup Y_{k-1} \supseteq \omega$ together with a $\Delta^1_1$ open largeness class $\Vcal \supseteq \bigcap_{\alpha < \wck} \Ucal_{C_\alpha}^{\Mcal_{\alpha}}$ such that $Y_i \notin \Ucal \cap \Vcal$ for every $i < k$. As each $Y_i$ is $\Delta^1_1$, there is some $i < k$ such that $C \subseteq^* Y_i$. Note also that since $C \in \bigcap_{\alpha < \wck} \Ucal_{C_\alpha}^{\Mcal_{\alpha}}$,  then $C \in \Lcal(\Vcal)$ and thus $\Lcal_C \cap \Vcal$ is a largeness class. It follows that $Y_j \in \Lcal_C \cap \Vcal$ for some $j < k$. As $j \neq i$ implies $|Y_j \cap C| < \infty$, then $Y_i \in \Lcal_C \cap \Vcal$ and thus $Y_i \in \Vcal$. As $\Lcal_C \cap \Ucal$ is a largeness class then by a similar argument, $Y_i \in \Lcal_C \cap \Ucal$ and thus $Y_i \in \Ucal$. It follows that $Y_i \in \Ucal \cap \Vcal$, contradicting our hypothesis. Thus $\Ucal \cap \bigcap_{\alpha < \wck} \Ucal_{C_\alpha}^{\Mcal_{\alpha}}$ is a largeness class.

Now from \Cref{prop-hyp-delta-min} we have that $\bigcap_{\alpha < \wck} \Ucal_{C_\alpha}^{\Mcal_{\alpha}}$ is minimal for $\Delta^1_1$ open sets, then $\bigcap_{\alpha < \wck} \Ucal_{C_\alpha}^{\Mcal_{\alpha}} \subseteq \Ucal$.
\end{proof}

\begin{definition}
Let $\Bcal = \bigcup_{\alpha < \wck} \Bcal_{\alpha}$ be a $\Sigma^0_{\wck}$ class. Let $p = (\sigma, X) \in \Pb_{\wck}$. We define $p \qvdash \Bcal$ if the set
$$\{Y\ :\ \exists \tau \subseteq Y - \{0, \dots, |\sigma|\}\ \exists \alpha < \wck\ \sigma \cup \tau \nqvdash 2^\omega - \Bcal_{\alpha} \} \cap \Lcal_C$$
is a largeness class.
\end{definition}

Given a $\Sigma^0_{\wck}$ class $\Bcal = \bigcup_{\alpha < \wck} \Bcal_{\alpha}$ the following set
$$\{Y\ :\ \exists \tau \subseteq Y - \{0, \dots, |\sigma|\}\ \exists \alpha < \wck\ \sigma \cup \tau \nqvdash 2^\omega - \Bcal_{\alpha} \}$$
is a $\Pi^1_1$ open set, that is an open set $\bigcup_{\sigma \in B} [\sigma]$ where $B = \bigcup_{\alpha < \wck} B_\alpha$ is a $\Pi^1_1$ set of strings. We also suppose that each $B_\alpha$ is $\halt^{(\alpha)}$-computable and that $\{B_\alpha\}_{\alpha <\wck}$ is increasing. Given such sets we write $\Ucal_\alpha$ for the $\Delta^1_1$ open set $\bigcup_{\sigma \in B_\alpha} [\sigma]$.

\begin{proposition}
Let $\Ucal$ be an upward-closed $\Pi^1_1$ open set. The class $\Ucal \cap \Lcal_C$ is a largeness class iff there exists some $\alpha < \wck$ such that $\Ucal_\alpha \cap \Lcal_C$ is a largeness class.
\end{proposition}
\begin{proof}
Suppose $\Ucal_\alpha \cap \Lcal_C$ is a largeness class. Then clearly $\Ucal \cap \Lcal_C$ is a largeness class. Suppose now that $\Ucal \cap \Lcal_C$ is a largeness class. For each $n$ let $\Ucal_n^C$ be the $\Sigma^0_1(C)$ open set such that $\Lcal_C = \bigcap_{n} \Ucal_n^C$. We have
$$\forall n\ \forall k\ \exists \alpha\ \forall Y_0 \cup \dots \cup Y_{k-1}\ \exists i < k\ \exists \sigma \subseteq Y_i\ [\sigma] \subseteq \Ucal_\alpha \cap \Ucal_n^C$$

Note that given $k$ and $\alpha$ the predicate $P^{n, k}_\alpha \equiv \forall Y_0 \cup \dots \cup Y_{k-1}\ \exists i < k\ \exists \sigma \subseteq Y_i\ [\sigma] \subseteq \Ucal_\alpha \cap \Ucal_n^C$ is $\Sigma^0_1(C \oplus \halt^{(\alpha+1)})$ uniformly in $n,k$ and $\alpha$. Thus the function $f:\omega^2 \rightarrow \wck$ which to $n,k$ associates the smallest $\alpha$ such that $P^{n,k}_\alpha$ is true is a total $\Pi^1_1(C)$ function. By $\Sigma^1_1$-boundedness we have $\beta = \sup_{n,k} f(n,k) < \omega_1^C = \wck$. It follows that
$$\forall n\ \forall k\ \forall Y_0 \cup \dots \cup Y_{k-1}\ \exists i < k\ \exists \sigma \subseteq Y_i\ [\sigma] \subseteq \Ucal_\beta \cap \Ucal_n^C$$
Also $\Ucal_\beta \subseteq \Ucal$ is such that $\Ucal_\beta \cap \Lcal_C$ is a largeness class.
\end{proof}

\begin{corollary} \label{cor-hyp-forcepii}
Let $\Bcal = \bigcup_{\alpha < \wck} \Bcal_{\alpha}$ be a $\Sigma^0_{\wck}$ class. Let $(\sigma, X) \in \Pb_{\wck}$. The relation $p \qvdash \Bcal$ is $\Sigma^0_{\wck}(C)$
\end{corollary}
\begin{proof}
The relation $p \qvdash \Bcal$ is equivalent to
$$\exists \alpha < \wck\ \{Y\ :\ \exists \tau \subseteq Y - \{0, \dots, |\sigma|\}\ \sigma \cup \tau \nqvdash 2^\omega - \Bcal_{\alpha} \} \cap \Lcal_C$$
is a largeness class
\end{proof}

\begin{corollary} \label{hyp-corpi11accclass}
The class $\bigcap_{\alpha < \wck} \Ucal_{C_\alpha}^{\Mcal_\alpha}$ is minimal for $\Pi^1_1$ open sets $\Ucal$ such that $\Ucal \cap \Lcal_C$ is a largeness class.
\end{corollary}
\begin{proof}
Given a $\Pi^1_1$-open set $\Ucal$ such that $\Ucal \cap \Lcal_C$, there must be $\alpha < \wck$ such that $\Ucal_\alpha \cap \Lcal_C$ is a largeness class. By \Cref{hyp-lemma-minC} it must be that $\bigcap_{\alpha < \wck} \Ucal_{C_\alpha}^{\Mcal_{\alpha}} \subseteq \Ucal_\alpha$.
\end{proof}

\begin{definition}
Let $\Bcal = \bigcap_{\alpha < \wck} \Bcal_{\alpha}$ be a $\Pi^0_{\wck}$ class. Let $p = (\sigma, X) \in \Pb_{\wck}$. We define $p \Vdash \Bcal$ if for every $\tau \subseteq X - \{0, \dots, |\sigma|\}$ and for every $\alpha < \wck$ we have $\sigma \cup \tau \qvdash \Bcal_{\alpha}$
\end{definition}

\begin{proposition} \label{hyp-sideprop-truth}
Let $\Bcal = \bigcap_{\alpha < \wck} \Bcal_{\alpha}$ be a $\Pi^0_{\wck}$ class. Let $\Fcal$ be sufficiently generic with $p \in \Fcal$. If $p \Vdash \Bcal$, then $G_{\Fcal} \in \Bcal$.
\end{proposition}
\begin{proof}
Using \Cref{prop-hyp-forcext}, for every $\alpha$ and every $q \leq p$, there is  some $r \leq q$ such that $r \Vdash \Bcal_{\alpha}$. Thus for every $\alpha$ the set $\{r\ :\ r \Vdash \Bcal_{\alpha}\}$ is dense below $p$. It follows from \Cref{th-hyp-forcingimpliestruth} that if $\Fcal$ is sufficiently generic, $G_{\Fcal} \in \Bcal$.
\end{proof}

\begin{definition}
Let $\Bcal = \bigcup_{n \in \omega} \Bcal_{n}$ be a $\Sigma^0_{\wck+1}$ class where each $\Pi^0_{\wck}$ set $\Bcal_n = \bigcap_{\alpha < \wck} \Bcal_{n, \alpha}$. We define
$p \qvdash \Bcal$ if the set
$$\{Y\ :\ \exists \tau \subseteq Y - \{0, \dots, |\sigma|\}\ \exists n\ \sigma \cup \tau \nqvdash 2^\omega - \Bcal_{n} \} \cap \Lcal_C$$
is a largeness class.
\end{definition}

Given a $\Sigma^0_{\wck+1}$ class $\Bcal = \bigcup_{n \in \omega} \Bcal_{n}$ with $\Bcal_n = \bigcap_{\alpha < \wck} \Bcal_{n, \alpha}$, the following set
$$\Ucal = \{Y\ :\ \exists \tau \subseteq Y - \{0, \dots, |\sigma|\}\ \exists n\ \sigma \cup \tau \nqvdash 2^\omega - \Bcal_{n} \}$$
is a $\Sigma^1_1(C)$ open set, that is an open set $\Ucal = \bigcup_{\sigma \in B} [\sigma]$ where $B = \bigcap_{\alpha < \wck} B_\alpha$ is a $\Sigma^1_1(C)$ set of strings. We furthermore assume that $\{B_\alpha\}_{\alpha < \wck}$ is decreasing. We then write $\Ucal_\alpha$ for the $\Delta^1_1(C)$-open set $\bigcup_{\sigma \in B_\alpha} [\sigma]$.

Computability theorists have a strong habits of working with enumerable open sets. With that respect, $\Sigma^1_1$-open sets, that is, co-enumerable along the computable ordinals, are strange objects to consider. Note that given such an open set we have $\Ucal \subseteq \bigcap_{\alpha < \wck} \Ucal_\alpha$, but not necessarily equality. However the elements $X$ of $\bigcap_{\alpha < \wck} \Ucal_\alpha - \Ucal$ are all such that $\omega_1^X > \wck$. It is in particular a meager and nullset.

Let us detail a little bit the set $B = \bigcap_{\alpha < \wck} B_\alpha$ that we can consider so that $\Ucal = \bigcup_{\sigma \in B} [\sigma]$. To ease the notation we introduce the following definition, in the same spirit as $\Ucal(\Bcal, \sigma)$ defined above:

\begin{definition} \label{def-hyp-mainstuff}
Let $\Bcal$ be a $\Sigma^0_{\alpha}$ class. We define
$\Vcal(\Bcal, \sigma)$ to be the set
$$\{Y\ :\ \exists \tau \subseteq Y - \{0, \dots, |\sigma|\}\ \sigma \cup \tau \nqvdash \Bcal \}$$
\end{definition}

Given a $\Sigma^0_{\wck+1}$ class $\Bcal = \bigcup_{n \in \omega} \Bcal_{n}$ with $\Bcal_n = \bigcap_{\alpha < \wck} \Bcal_{n, \alpha}$, given
$$\Ucal = \{Y\ :\ \exists \tau \subseteq Y - \{0, \dots, |\sigma|\}\ \exists n\ \sigma \cup \tau \nqvdash 2^\omega - \Bcal_{n} \}$$
we have by \Cref{cor-hyp-forcepii} that $\Ucal$ equals:
$$\{Y\ :\ \exists \tau \subseteq Y - \{0, \dots, |\sigma|\}\ \exists n\ \forall \alpha < \wck\ \Vcal(2^\omega - \Bcal_{n, \alpha}, \sigma \cup \tau) \cap \Lcal_C \text{ is not a largeness class}\}$$
Let
$$B = \{\tau\ :\ \exists n\ \forall \alpha < \wck\ \Vcal(2^\omega - \Bcal_{n, \alpha}, \sigma \cup \tau) \cap \Lcal_C \text{ is not a largeness class} \}$$
Let
$$B_\alpha = \{\tau\ :\ \exists n\ \forall \beta < \alpha\ \Vcal(2^\omega - \Bcal_{n, \beta}, \sigma \cup \tau) \cap \Lcal_C \text{ is not a largeness class} \}$$
By $\Sigma^1_1$-boundedness we have that $B = \bigcap_{\alpha} B_\alpha$. We also have $\Ucal = \bigcup_{\sigma \in B} [\sigma]$.

We now show the core lemma that will be used to show $\omega_1^{G_{\Fcal}} = \wck$ for $\Fcal$ a sufficiently generic filter:

\begin{lemma} \label{hyp-lem-sigma1}
Let $B = \bigcap_{\alpha < \wck} B_\alpha$ be a $\Sigma^1_1(C)$ set of strings where each $B_\alpha$ is $\Delta^1_1(C)$ uniformly in $\alpha$ and where $\beta < \alpha$ implies $B_\alpha \subseteq B_\beta$. Let $\Ucal = \bigcup_{\sigma \in B} [\sigma]$ be a $\Sigma^1_1(C)$ upward closed open set with $\Ucal_\alpha = \bigcup_{\sigma \in B_\alpha} [\sigma]$ be a $\Delta^1_1(C)$ upward closed open set. We have $\Ucal \subseteq \bigcap_{\alpha < \wck} \Ucal_\alpha$. Furthermore, $\Ucal \cap \Lcal_C$ is a largeness class iff for every $\alpha < \wck$ , $\Ucal_\alpha \cap \Lcal_C$ is a largeness class.
\end{lemma}
\begin{proof}
It is clear that $\Ucal \subseteq \bigcap_{\alpha < \wck} \Ucal_\alpha$. Also it is clear that if $\Ucal \cap \Lcal_C$ is a largeness class, then also $\bigcap_{\alpha < \wck} \Ucal_{\alpha} \cap \Lcal_C$ is a largeness class.

Suppose $\Ucal \cap \Lcal_C$ is not a largeness class. Then there is a cover $Y_0 \cup \dots \cup Y_{k - 1} \supseteq \omega$ with $Y_i \notin \Ucal \cap \Lcal_C$ for every $i < k$. There must be a $\Sigma^0_1(C)$ open set $\Vcal$ such that $Y_i \notin \Ucal \cap \Vcal$ for every $i \leq k$.

Let $f:\omega \rightarrow \wck$ be the function which on $n$ finds a cover $\sigma_0 \cup \dots \cup \sigma_k \supseteq \{0, \dots, n\}$ and $\alpha$ such that for $i < k$ and every $\tau \preceq \sigma_i$ we have $[\tau] \subseteq \Vcal$ implies $\tau \notin B_\alpha$. As $\Ucal \cap \Vcal$ is not a largeness class, $f$ is a total $\Pi^1_1(C)$ function. By $\Sigma^1_1$-boundedness, $\beta = \sup_n f(n) < \omega_1^C = \wck$. By compactness, there is a cover $Y_0 \cup \dots \cup Y_{k-1}$ such that for every $i < k$ if $Y_i \in \Vcal$ then for every $\tau \prec Y_i$, $\tau \notin B_\beta$ and thus $Y_i \notin \Ucal_\beta$.

It follows that $\Ucal_\beta \cap \Lcal_C$ is not a largeness class.
\end{proof}

\begin{corollary}
$\Lcal_C$ contains a unique largeness subclass, which is minimal for both $\Pi^1_1$ and $\Sigma^1_1(C)$-open sets $\Ucal$.

%
\end{corollary}
\begin{proof}
Suppose $\Ucal_0, \Ucal_1$ are two $\Sigma^1_1(C)$ open sets with $\Ucal_i = \bigcup_{\sigma \in B_i} [\sigma]$ and $\Ucal_{i, \alpha} = \bigcup_{\sigma \in B_{i, \alpha}} [\sigma]$. for $i < 2$. Suppose also $\Ucal_0 \cap \Lcal_{C}$ and $\Ucal_1 \cap \Lcal_{C}$ are largeness classes. By \Cref{hyp-lem-sigma1} it follows that $\bigcap_{\alpha < \wck} \Ucal_{0, \alpha} \cap \Lcal_C$ and $\bigcap_{\alpha < \wck} \Ucal_{1, \alpha} \cap \Lcal_C$ are largeness classes. By \Cref{hyp-lemma-minC} it follows that $\bigcap_{\alpha < \wck} \Ucal_{C_\alpha}^{\Mcal_{\alpha}} \subseteq \bigcap_{\alpha < \wck} \Ucal_{0, \alpha}$ and $\bigcap_{\alpha < \wck} \Ucal_{C_\alpha}^{\Mcal_{\alpha}} \subseteq \bigcap_{\alpha < \wck} \Ucal_{1, \alpha}$.

Thus $\bigcap_{\alpha < \wck} \Ucal_{0, \alpha} \cap \bigcap_{\alpha < \wck} \Ucal_{1, \alpha} = \bigcap_{\alpha < \wck} (\Ucal_{0, \alpha} \cap \Ucal_{1, \alpha})$ is a largeness class and thus by \Cref{hyp-lem-sigma1} we have that $\Ucal_0 \cap \Ucal_1$ is a largeness class.

It follow that the intersection $\Ical$ of every $\Sigma^1_1(C)$ open set $\Ucal$ such that $\Ucal \cap \Lcal_C$ is a largeness class, is a largeness class. Furthermore as $\Ucal_{C_\alpha}^{\Mcal_\alpha} \cap \Lcal_C$ is a largeness class for every $\alpha$, the class $\Ical$ must be included in $\bigcap_{\alpha < \wck} \Ucal_{C_\alpha}^{\Mcal_\alpha}$. Also from \Cref{hyp-corpi11accclass} the class $\bigcap_{\alpha < \wck} \Ucal_{C_\alpha}^{\Mcal_\alpha}$ is minimal for $\Pi^1_1$-open sets $\Ucal$ such that $\Ucal \cap \Lcal_C$ is a largeness class. It follows that the class $\Ical \cap \Lcal_C$ is minimal for $\Sigma^1_1(C)$ and $\Pi^1_1$ open sets.
\end{proof}

We can now detail the class $\Scal$ involved in the definition of $\Pb_{\wck}$ : Let $\Scal$ be the unique largeness class included in $\Lcal_C$ which is minimal for $\Sigma^1_1(C)$ and $\Pi^1_1$ open sets. Note that $\Scal$ must be partition regular.

\begin{lemma} \label{hyp-collapse1}
Consider a $\Sigma^0_{\wck+1}$ class $\Bcal = \bigcup_{n \in \omega} \Bcal_{n}$ with $\Pi^0_{\wck}$ set $\Bcal_n = \bigcap_{\alpha \in \wck} \Bcal_{n, \alpha}$. Let $p = (\sigma, X) \in \Pb_{\wck}$.
Suppose $\sigma \qvdash \Bcal$. Then there is a condition $q \leq p$ together with some $n$ such that $q \Vdash \bigcap_{\alpha < \wck} \Bcal_{n, \alpha}$
\end{lemma}
\begin{proof}
Let
$$\Ucal = \{Y\ :\ \exists \tau \subseteq Y - \{0, \dots, |\sigma|\}\ \exists n\ \sigma \cup \tau \nqvdash 2^\omega -  \Bcal_{n} \}$$
The class $\Ucal$ is a $\Sigma^1_1(C)$-open set and $\Ucal \cap \Lcal_C$ is a largeness class. As $\Scal$ is minimal for $\Sigma^1_1(C)$-open sets, $\Scal \subseteq \Ucal$. As $X \in \Scal \subseteq \Ucal$. Then there is some $\tau \subseteq X - \{0, \dots, |\sigma|\}$ and some $n$ such that
$\sigma \cup \tau \nqvdash 2^\omega - \Bcal_{n}$.
Let now
$$\Vcal = \{Y\ :\ \exists \rho \subseteq Y - \{0, \dots, |\sigma \cup \tau|\}\ \exists \alpha\ \sigma \cup \tau \cup \rho \nqvdash \Bcal_{n, \alpha}\}$$
As $\sigma \cup \tau \nqvdash \bigcup_{\alpha \in \wck} 2^\omega -  \Bcal_{n, \alpha}$ then $\Vcal \cap \Lcal_C$ is not a largeness class. Thus there is a cover $Y_0 \cup \dots \cup Y_{k-1} = \omega$ such that $Y_i \notin \Vcal \cap \Lcal_C$ for every $i < k$. As $\Vcal \cap \Lcal_C$ is upward-closed, $X \cap Y_i \notin \Vcal \cap \Lcal_C$ for every $i < k$. As $\Scal \subseteq \Lcal_C$ is partition regular, there is some $i < k$ such that $X \cap Y_i \in \Scal \subseteq \Lcal_C$. Therefore we must have $X \cap Y_i \notin \Vcal$ and thus
$$\forall \rho \subseteq X \cap Y_i - \{0, \dots, |\sigma \cup \tau|\}\ \forall \alpha\ \sigma \cup \tau \cup \rho \qvdash \Bcal_{n, \alpha}$$
Thus $(\sigma \cup \tau, X \cap Y_i)$ is an extension of $(\sigma, X)$ such that:
$$(\sigma \cup \tau, X \cap Y_i) \Vdash \bigcap_{\alpha < \wck} \Bcal_{n, \alpha}$$
\end{proof}

\begin{lemma} \label{hyp-collapse2}
Consider a $\Sigma^0_{\wck+1}$ class $\Bcal = \bigcup_{n \in \omega} \Bcal_{n}$ with $\Pi^0_{\wck}$ set $\Bcal_n = \bigcap_{\alpha < \wck} \Bcal_{n, \alpha}$. Let $p = (\sigma, X) \in \Pb_{\wck}$.
Suppose $\sigma \nqvdash \Bcal$. Then there is a condition $q \leq p$ together with some $\beta < \wck$ such that $q \Vdash \bigcap_{n \in \omega} \bigcup_{\alpha < \beta} 2^\omega - \Bcal_{n, \alpha}$
\end{lemma}
\begin{proof}
Let
$$\Ucal = \{Y\ :\ \exists \tau \subseteq Y - \{0, \dots, |\sigma|\}\ \exists n\ \sigma \cup \tau \nqvdash 2^\omega -  \Bcal_{n} \}$$
The class $\Ucal$ is a $\Sigma^1_1(C)$-open set and $\Ucal \cap \Lcal_C$ is not a largeness class. Let us recall \Cref{def-hyp-mainstuff} together with the notation coming after it: $\Vcal(\Bcal, \sigma)$ is the set
$$\{Y\ :\ \exists \tau \subseteq Y - \{0, \dots, |\sigma|\}\ \sigma \cup \tau \nqvdash \Bcal \}$$
Together with
$$B = \{\tau\ :\ \exists n\ \forall \alpha < \wck\ \Vcal(\Bcal_{n, \alpha}, \sigma \cup \tau) \cap \Lcal_C \text{ is not a largeness class} \}$$
with $B = \bigcap_{\alpha < \wck} B_\alpha$ such that
$$B_\alpha = \{\tau\ :\ \exists n\ \forall \beta < \alpha\ \Vcal(\Bcal_{n, \beta}, \sigma \cup \tau) \cap \Lcal_C \text{ is not a largeness class} \}$$
and with $\Ucal = \bigcup_{\sigma \in B} [\sigma]$.

Using \Cref{hyp-lem-sigma1}, there is some $\alpha < \wck$ such that the set
$$\Ucal_\alpha = \{Y\ :\ \exists \tau \subseteq Y - \{0, \dots, |\sigma|\}\ \exists n\ \forall \beta < \alpha\ \Vcal(\Bcal_{n, \beta}, \sigma \cup \tau) \cap \Lcal_C \text{ is not a largeness class}\}$$

%

is such that $\Ucal_\alpha \cap \Lcal_C$ is not a largeness class. Thus there is a cover $Y_0 \cup \dots \cup Y_{k-1} \supseteq \omega$ such that $Y_i \notin \Ucal_\alpha \cap \Lcal_C$ for every $i < k$. As $\Ucal_\alpha \cap \Lcal_C$ is upward-closed, then also $X \cap Y_i \notin \Ucal_\alpha \cap \Lcal_C$ for every $i < k$. As $X \in \Scal \subseteq \Lcal_C$ and as $\Scal$ is partition regular, there is some $i < k$ such that $X \cap Y_i \in \Scal \subseteq \Lcal_C$. It follows that $X \cap Y_i \notin \Ucal_\alpha$ and thus that:
$$\forall \tau \subseteq X \cap Y_i - \{0, \dots, |\sigma|\}\ \forall n\ \exists \beta < \alpha\ \Vcal(\Bcal_{n, \beta}, \sigma \cup \tau) \cap \Lcal_C \text{ is a largeness class}$$
Let $\{\beta_m\}_{m \in \omega}$ be such that $\sup_m \beta_m = \alpha$. Let $\tau \subseteq Y - \{0, \dots, |\sigma|\}$ and $n \in \omega$. We have for some $m$ that $\Vcal(\Bcal_{n, \beta_m}, \sigma \cup \tau) \cap \Lcal_C$ is a largeness class. Then the set
$$\{Y\ :\ \exists \rho \subseteq Y - \{0, \dots, |\sigma|\}\ \exists m\ \sigma \cup \tau \cup \rho \nqvdash \Bcal_{n, \beta_m} \} \cap \Lcal_C$$
is a largeness class and then
$$\{Y\ :\ \exists \rho \subseteq Y - \{0, \dots, |\sigma|\}\ \exists m\ \sigma \cup \tau \cup \rho \nqvdash \Bcal_{n, \beta_m}\} \cap \Ucal_{C_\alpha}^{\Mcal_\alpha}$$
is a largeness class and thus that $\sigma \cup \tau \qvdash \bigcup_{m} 2^\omega - \Bcal_{n, \beta_m}$. As this is true for every $n$ and every $\tau \subseteq Y - \{0, \dots, |\sigma|\}$ it follows that $(\sigma, X \cap Y_i)$ is an extension of $(\sigma, X)$ such that
$$(\sigma, X \cap Y_i) \Vdash \bigcap_{n \in \omega} \bigcup_{\beta < \alpha} 2^\omega - \Bcal_{n, \beta}$$
\end{proof}

We now show that if $\Fcal \subseteq \Pb_{\wck}$ is sufficiently generic, then $\omega_1^{G_\Fcal} = \wck$. We use the following fact : If $\omega_1^G > \wck$, then in particular some $G$-computable ordinal must code for $\wck$, that is, there must be a $G$-computable function $\Phi$ such that for every $n$, $\Phi(G, n)$ codes, relative to $G$, for an ordinal smaller than $\wck$ and with $\sup_n |\Phi(G, n)| = \wck$. We show that this never happens by forcing that for every functional $\Phi$ either for some $n$, $\Phi(G, n)$ does not code for an ordinal smaller than $\wck$, or there is an ordinal $\alpha < \wck$ such that $\Phi(G, n)$ always codes for some ordinal smaller than $\alpha$.

Given $G$ and $\alpha$ let $\Ocal^G_\alpha$ be the set of $G$-codes for ordinals smaller than $\alpha$. For $\alpha < \wck$, the class $\{G\ :\ n \in \Ocal^G_\alpha\}$ is $\Delta^1_1$ uniformly in $\alpha$ and $n$.

\begin{theorem} \label{th-wck}
Suppose $\Fcal \subseteq \Pb_{\wck}$ is sufficiently generic. Then $\omega_1^{G_\Fcal} = \wck$
\end{theorem}
\begin{proof}
Let $p \in \Pb_{\wck}$ be a condition. Given a functional $\Phi : 2^\omega \times \omega \rightarrow \omega$, let
$$\Bcal = \{X\ :\ \exists n\ \forall \alpha < \wck\ \Phi(X, n) \notin \Ocal_\alpha^X\}$$
Suppose $p \qvdash \Bcal$. Then from \Cref{hyp-collapse1}, there is an extension $q \leq p$ and some $n$ such that
$$q \Vdash \{X\ :\ \forall \alpha < \wck\ \Phi(X, n) \notin \Ocal_\alpha^X\}$$
It follows from \Cref{hyp-sideprop-truth} that if $\Fcal$ is sufficiently generic for every $\alpha < \wck$, $\Phi(G_{\Fcal}, n) \notin \Ocal_\alpha^{G_{\Fcal}}$.
Suppose now $p \nqvdash \Bcal$. Then from \Cref{hyp-collapse2}, there is an extension $q \leq p$ and some $\alpha < \wck$ such that
$$q \Vdash \{X\ :\ \forall n\ \Phi(X, n) \in \Ocal_\alpha^X\}$$
It follows from \Cref{th-hyp-forcingimpliestruth} that if $\Fcal$ is sufficiently generic, $\sup_n \Phi(G_{\Fcal}, n) \leq \alpha$.
\end{proof}

We can finally deduce our final theorem

\begin{theoremnonumber}[Main theorem 4 (Theorem \ref{maintheorem4})]
Let $B$ be non hyperarithmetical. Every set $A$ has an infinite subset $H \subseteq A$ or $H \subseteq \overline{A}$ such that $B$ is not hyperarithmetical in $H$, in particular with $\omega_1^H = \wck$.
\end{theoremnonumber}
\begin{proof}
By combining \Cref{th-wck} together with \Cref{th-alphaconeavoid}
\end{proof}

\vspace{0.5cm}

\bibliographystyle{plain}
\bibliography{bibliography}

\end{document}